\documentclass[reqno]{amsart}
\usepackage{amssymb}
\usepackage{hyperref}

 \newtheorem{Theorem}{Theorem}[section]
 \newtheorem{Corollary}[Theorem]{Corollary}
 \newtheorem{Lemma}[Theorem]{Lemma}
 \newtheorem{Proposition}[Theorem]{Proposition}

 \newtheorem{Definition}[Theorem]{Definition}
\newtheorem{Question}[Theorem]{Question}

 \newtheorem{Remark}[Theorem]{Remark}

 \numberwithin{equation}{section}

\begin{document}

\title{Boundary points, Minimal $L^{2}$ integrals and Concavity property}
\author{Shijie Bao}
\address{Shijie Bao: Institute of Mathematics, Academy of Mathematics and Systems Science, Chinese
	Academy of Sciences, Beijing 100190, China}
\email{bsjie@amss.ac.cn}

\author{Qi'an Guan}
\address{Qi'an Guan: School of
Mathematical Sciences, Peking University, Beijing 100871, China.}
\email{guanqian@math.pku.edu.cn}

\author{Zheng Yuan}
\address{Zheng Yuan: Institute of Mathematics, Academy of Mathematics and Systems Science, Chinese
	Academy of Sciences, Beijing 100190, China.}
\email{yuanzheng@amss.ac.cn}

\subjclass[2020]{32D15 32E10 32L10 32U05 32W05}

\thanks{}

\keywords{concavity, minimal $L^2$ integral, Jonsson-Musta\c{t}\u{a}'s conjecture, strong openness conjecture}

\date{\today}

\dedicatory{}

\commby{}


\begin{abstract}
For the purpose of proving the strong openness conjecture of multiplier ideal sheaves, Jonsson-Musta\c{t}\u{a} posed an enhanced conjecture and proved the two-dimensional case, which says that: the Lebesgue measure of the set $\big\{c_o^F(\psi)\psi-\log|F|<\log r\big\}$ divided by $r^2$ has a uniform positive lower bound independent of $r$, for a plurisubharmonic function $\psi$ and a holomorphic function $F$ near the origin $o$. Jonsson-Musta\c{t}\u{a}'s conjecture was proved by Guan-Zhou depending on the truth of the strong openness conjecture. However, it is still a question whether one can prove Jonsson-Musta\c{t}\u{a}'s conjecture without using the strong openness property, and obtain a sharp effectiveness result for this conjecture.

In this article, we use an $L^2$ method with the weight functions $\psi-\log|F|$ and firstly consider a module at at a boundary point of the sublevel sets of a plurisubharmonic function. By studying the minimal $L^{2}$ integrals on the sublevel sets of a plurisubharmonic function with respect to the module at the boundary point, we establish a concavity property of the minimal $L^{2}$ integrals. As applications, we obtain a sharp effectiveness result related to Jonsson-Musta\c{t}\u{a}'s conjecture, which completes the approach from the conjecture to the strong openness property. We also obtain a strong openness property of the module and a lower semi-continuity property with respect to the module.
\end{abstract}

\maketitle

\tableofcontents

\section{Introduction}

\subsection{Backgrounds and Motivations}
\

Let $\psi$ be a plurisubharmonic function of a complex manifold $M$ (see \cite{Demaillybook}).
Recall that multiplier ideal sheaf $\mathcal{I}(\psi)$ is the sheaf of germs of holomorphic functions $f$ such that $|f|^2e^{-\psi}$ is locally integrable,
which was widely discussed in the study of several complex variables, complex algebraic geometry and complex differential geometry
(see e.g. \cite{Tian,Nadel,Siu96,DEL,DK01,DemaillySoc,DP03,Lazarsfeld1,Lazarsfeld2,Siu05,Siu09,DemaillyAG,Guenancia}).

The strong openness property of multiplier ideal sheaves, i.e. 
\[\mathcal{I}(\psi)=\mathcal{I}_+(\psi):=\bigcup_{\epsilon>0}\mathcal{I}((1+\epsilon)\psi),\]
is an important feature of multiplier ideal sheaves,
 and ``opened the door to new types of approximation techniques" (see \cite{M-V15})
(see e.g. \cite{GZSOC,K16,cao17,cdM17,FoW18,DEL18,ZZ2018,GZ20,berndtsson20,ZZ2019,ZhouZhu20siu's,FoW20,KS20,DEL21}).
The strong openness property was conjectured by Demailly \cite{DemaillySoc}, and proved by Guan-Zhou \cite{GZSOC} (the 2-dimensional case was proved by Jonsson-Musta\c{t}\u{a} \cite{JM12}). After that, Guan-Zhou \cite{GZeff} established an effectiveness result of the strong openness property by considering the minimal $L^{2}$ integral on the pseudoconvex domain $D$.

When $\mathcal{I}(\psi)=\mathcal{O}$, the strong openness property degenerates to the openness property, which was conjectured by Demailly-Koll\'ar \cite{DK01}
and proved by Berndtsson \cite{Berndtsson2} (the 2-dimensional case was proved by Favre-Jonsson in \cite{FavreJonsson}).

Let $D$ be a pseudoconvex domain in $\mathbb{C}^n$ containing the origin $o\in\mathbb{C}^n$, and let $\psi$ be a plurisubharmonic function on $D$. For a holomorphic function $h$ on a neighborhood of $o$, recall that
\[c_o^h(\psi):=\sup\big\{c\ge0:|h|^2e^{-2c\psi} \ \text{is} \ L^1 \ \text{on a neighborhood of} \ o\big\}\]
is the jumping number (see \cite{JM13}).  
When $h\equiv 1$, $c_o^h(\psi)$ degenerates to the complex singularity exponent $c_o(\psi)$ (see \cite{Tian,DK01}).

In \cite{DK01} (see also \cite{JM13}), Demailly and Koll\'{a}r conjectured the following more precise form of the openness property:

\

\textbf{Conjecture D-K:} 
\emph{If $c_o(\psi)<+\infty$, $\frac{1}{r^2}\mu(\{c_o(\psi)\psi<\log r\})$ has a uniform positive lower bound independent of $r\in(0,1)$, where $\mu$ is the Lebesgue measure on $\mathbb{C}^n$.}

\

The 2-dimensional case of Conjecture D-K was proved by Favre-Jonsson in \cite{FavreJonsson},
which deduces the 2-dimensional case of the openness property.
Depending on the openness property, Guan-Zhou \cite{GZeff} proved Conjecture D-K. 

In \cite{G16}, Guan gave a sharp effectiveness result related to Conjecture D-K independent of the openness property, which completed the approach from Conjecture D-K to the openness property.

In \cite{xu}, Xu completed the algebraic approach to Conjecture D-K.

In order to prove the strong openness property, Jonsson and Musta\c{t}\u{a} (see \cite{JM13}, see also \cite{JM12}) posed the following conjecture, and proved the 2-dimensional case \cite{JM12}:

\

\textbf{Conjecture J-M:}  \emph{If $c_o^F(\psi)<+\infty$ for a holomorphic function $F$ on $D$, then $\frac{1}{r^2}\mu(\{c_o^F(\psi)\psi-\log|F|<\log r\})$ has a uniform positive lower bound independent of $r\in(0,1)$, where $\mu$ is the Lebesgue measure on $\mathbb{C}^n$.}

\

Depending on the strong openness property, Guan-Zhou \cite{GZeff} proved Conjecture J-M:
\begin{Theorem}[\cite{GZeff}]\label{thm:JM.GZ's_result}
	If $c_o^F(\psi)<+\infty$ and $\sup_{D}e^{(1+\delta)\max\{2c_o^{F}(\psi)\psi,2\log|F|\}}<+\infty$, then
	\begin{displaymath}
		\begin{split}
			&\liminf_{r\rightarrow0+0}\frac{1}{r^2}\mu(\{c_o^F(\psi)\psi-\log|F|<\log r\})\\
			\ge& \sup_{\delta\in\mathbb{Z}_{>0}}\frac{C_{F^{1+\delta},2c_o^F(\psi)\psi+\delta\max\{2c_o^{F}(\psi)\psi,2\log|F|\}}(o)}{\left(1+\frac{1}{\delta}\right)\sup_{D}e^{(1+\delta)\max\{2c_o^{F}(\psi)\psi,2\log|F|\}}},
		\end{split}
	\end{displaymath}
where for a plurisubharmonic function $\phi$ and a holomorphic function $F_0$ on $D$,
\[C_{F_0,\phi}(o):=\inf\left\{\int_D|\tilde F|^2:\tilde F\in\mathcal{O}(D) \ \& \ (\tilde F-F_0,o)\in \mathcal{I}(\phi)_o \right\}.\]
\end{Theorem}

A natural question is:
\begin{Question}
\label{Q:JM}
Can one obtain a sharp effectiveness result related to Conjecture J-M \textbf{independent} of the strong openness property,
which completes the approach from Conjecture J-M to the strong openness property?
\end{Question}

We give an affirmative answer to Question \ref{Q:JM} in the present paper.

\subsubsection{Formulations}
\

Recall that considering the minimal $L^{2}$ integrals on the sublevel sets of a given plurisubharmonic weight $\psi$, Guan \cite{G16}  established a concavity property of the minimal $L^2$ integrals on the sublevel sets of $\psi$ (see also \cite{G2018,GM,GM_Sci,GY-concavity,GMY-concavity2,GY-support}), and gave a sharp effectiveness result related to Conjecture D-K independent of the openness property.

The minimal $L^{2}$ integrals on the sublevel sets of the plurisubharmonic function $\psi$ considered in \cite{G16} (see also \cite{G2018,GM,GM_Sci,GY-concavity,GMY-concavity2,GY-support}) are respect to the modules(ideals) at the inner points of the sublevel sets.

However, in Conjecture J-M, we need to deal with the weight function $\psi-\log|F|$, which is not plurisubharmonic on the entire domain $D$, instead of a purely plurisubharmonic function $\psi$. In addition, the considered singular point (which is the origin $o$ in Conjecture J-M) can be not an inner point but a boundary point of the sublevel sets of the weight function $\psi-\log|F|$. 

To cross the obstacles mentioned above, we need an (optimal) $L^2$ estimate result with the weight function $\psi-\log|F|$, and we also need a module at the boundary point to replace the (multiplier) ideal at the inner point. Motivated the $L^2$ methods in \cite{guan-zhou13ap,GZeff,G16}, we give an $L^2$ method with the weight functions $\psi-\log|F|$ in Section \ref{sec:L2.methods}, where the technical part of the proof is given in the Appendix. 
In section \ref{sec:module}, we firstly discuss a module $I(a\Psi)$ at the boundary point and give some useful properties of the the module.

After these preparations, we consider the minimal $L^{2}$ integrals with respect to a module at a boundary point of the sublevel sets related to a function in the shape of $\psi-\log|F|$, and we obtain a concavity property of the minimal $L^{2}$ integrals. Then using the concavity property, we obtain a sharp effectiveness result of Conjecture J-M. 

\emph{It is worth noting that the effectiveness result (Corollary \ref{c:1}) of Conjecture J-M in the present paper implies Theorem \ref{thm:JM.GZ's_result} (see Remark \ref{r:1}), and its proof does not depend on the truth of the strong openness property.} Consequently, the present paper completes the approach from Conjecture J-M to the strong openness property moreover.

Using the concavity property of the minimal $L^{2}$ integrals with respect to a module at a boundary point, we also obtain a strong openness property of the module and a lower semi-continuity property with respect to the module.

\subsection{Concavity property}\label{main}
\

Let $F$ be a holomorphic function on a pseudoconvex domain $D\subset\mathbb{C}^n$ containing the origin $o\in\mathbb{C}^n$, and let $\psi$ be a plurisubharmonic function on $D$. Denote
$$\Psi:=\min\{\psi-2\log|F|,0\}.$$
If $F(z)=0$ for $z\in D$, we set $\Psi(z)=0$. Note that $\Psi=\psi+2\log|1/F|$ is a plurisubharmonic function on $\{\Psi<0\}$. Note that, \emph{if $F(o)=0$, then $o$ can be a boundary point of the sublevel set $\{\Psi<-t\}$ for any $t>0$}.

Set
\[\tilde J(\Psi)_o:=\big\{f\in\mathcal{O}(\{\Psi<-t\}\cap V): t\in\mathbb{R} \text{\ and} \ V \text{\ is a neighborhood of} \ o\big\}.\]
We define the following \textbf{module at the boundary point $o$} of the sublevel sets $\{\Psi<-t\}=\{\psi-2\log|F|<-t\}$,
where $t\geq0$. 
\begin{Definition}
	Denote by $J(\Psi)_o$ the set of the equivalence classes $f_o$ of $f\in \tilde J(\Psi)_o$ under the equivalence relation $f\sim g$ if $f=g$ on $\{\Psi<-t\}\cap V$, where $t\gg1$ and $V$ is some neighborhood of $o$.
\end{Definition}
If $o\in\bigcap_{t>0}\{\Psi<-t\}$, then $J(\Psi)_o$ equals to $\mathcal{O}_{\mathbb{C}^n,o}$, and $f_o$ is the germ $(f,o)$ of a holomorphic function $f$.

For any $f_o\in J(\Psi)_o$, $\tilde{f}_o\in J(\Psi)_o$ and $(h,o)\in\mathcal{O}_{\mathbb{C}^n,o}$, define
\[f_o+\tilde f_o=(f+\tilde f)_o, \ \text{and} \ (h,o)\cdot f_o=(hf)_o.\]
Then it is clear that $J(\Psi)_o$ is an $\mathcal{O}_{\mathbb{C}^n,o}$-module. 

\begin{Definition}
For any $h_o\in J(\Psi)_o$ and any $a\ge0$, we call $h_o\in I(a\Psi)_o$ if and only if there exist $t\gg1$ and a neighborhood $V$ of $o$ such that $\int_{\{\Psi<-t\}\cap V}|h|^2e^{-a\Psi}<+\infty$.
\end{Definition}

It is clear that $I(a\Psi)_o$ is an $\mathcal{O}_{\mathbb{C}^n,o}-$submodule of $J(\Psi)_o$. Denote $I_o:=I(0\Psi)_o$.

Let $f$ be a holomorphic function on $V_0\cap\{\Psi<-t_0\}$, where $V_0$ is a neighborhood of $o$ and $t_0>0$. Let $J$ be an $\mathcal{O}_{\mathbb{C}^n,o}-$submodule of $I_o$ such that $I(\Psi)_o\subset J$. 

\begin{Definition}
We define the \textbf{minimal $L^{2}$ integrals on the sublevel sets $\{\Psi<-t\}=\{\psi-2\log|F|<-t\}$ with respect to $J$} (which is a module at the boundary point $o$ of the sublevel sets) as follows:
\begin{displaymath}
 	G(t;\Psi,J,f):=\inf\left\{\int_{\{\Psi<-t\}}|\tilde f|^2:\tilde f\in\mathcal{O}(\{\Psi<-t\}) \ \& \ (\tilde f-f)_o\in J\right\},
 \end{displaymath}
where $t\in[0,+\infty)$.
\end{Definition}

In the present article, we establish the following concavity of $G(-\log r;\Psi,J,f)$.

\begin{Theorem}
	\label{thm:main}
	If there exists some $t_0\ge 0$ such that $G(t_0;\Psi,J,f)<+\infty$, then $G(-\log r;\Psi,J,f)$ is concave with respect to $r\in(0,1]$ and $\lim_{t\rightarrow+\infty}G(t;\Psi,J,f)=0$.
\end{Theorem}
When $F\equiv1$ and $\psi(o)=-\infty$, Theorem \ref{thm:main} can be referred to \cite{G16} (see also \cite{G2018,GM,GM_Sci,GY-concavity,GMY-concavity2,GY-support}). We also give an example in Appendix (Section \ref{sec:eg}) for the case that $F$ is not a constant function.

\subsection{Applications}

\

In this section, we present some applications of Theorem \ref{thm:main}.

 \subsubsection{A sharp effectiveness result related to Jonsson-Musta\c{t}\u{a}'s conjecture}
\

Let $f$ be a holomorphic function on $D$. Denote
$$\Psi_1:=\min\big\{2c_o^{fF}(\psi)\psi-2\log|F|,0\big\},$$
and
$$I_+(\Psi_1)_o:=\bigcup_{a>1}I(a\Psi_1)_o.$$

Recall the definition of the minimal $L^{2}$ integrals $G$ in Section \ref{main}. Theorem \ref{thm:main} implies the following result independent of the strong openness property.

\begin{Theorem}
	\label{thm:J_M} If $c_o^{fF}(\psi)<+\infty$, then
	\begin{displaymath}
		\frac{1}{r^2}\int_{\{c_o^{fF}(\psi)\psi-\log|F|<\log r\}}|f|^2\ge G(0; \Psi_1,I_+(\Psi_1)_o,f)>0
	\end{displaymath}
	holds for any $r\in(0,1]$.
\end{Theorem}

The constant $G(0; \Psi_1,I_+(\Psi_1)_o,f)$ is sharp. Let $D=\Delta\subset\mathbb{C}$ be the unit disc, and let $\psi=\log|z|$. Let $F\equiv1$, and let $f\equiv1$. It is clear that $c_o^{fF}(\psi)=1$, $\int_{\{c_o^{fF}(\psi)\psi-\log|F|<\log r\}}|f|^2=\pi r^2$ and $G(0;\Psi_1,I_+(\Psi_1)_o,f)=\pi$, which implies the sharpness of Theorem \ref{thm:J_M}.  A boundary point example can be seen in Appendix (Section \ref{sec:eg}).

When $f\equiv1$, Theorem \ref{thm:J_M} is the following sharp effectiveness result related to Conjecture J-M.

\begin{Corollary}\label{c:1}
	If $c_o^{F}(\psi)<+\infty$, then
	\begin{displaymath}
		\frac{1}{r^2}\mu\{c_o^{F}(\psi)\psi-\log|F|<\log r\}\ge G(0;\Psi_1,I_+(\Psi_1)_o,1)>0
	\end{displaymath}
	holds for any $r\in(0,1]$, where $\Psi_1=\min\{2c_o^F(\psi)\psi-2\log|F|,0\}$.
\end{Corollary}
The proof of Corollary \ref{c:1} is independent of the strong openness property, and Corollary \ref{c:1} recovers Theorem \ref{thm:JM.GZ's_result} due to the following remark.

\begin{Remark}\label{r:1}
	For any $\delta\in\mathbb{Z}_{>0}$, it holds that
	\[G(0;\Psi_1,I_+(\Psi_1)_o,1)\ge \frac{C_{F^{1+\delta},2c_o^F(\psi)\psi+\delta\max\{2c_o^{F}(\psi)\psi,2\log|F|\}}(o)}{\left(1+\frac{1}{\delta}\right)\sup_{D}e^{(1+\delta)\max\{2c_o^{F}(\psi)\psi,2\log|F|\}}},\]
	hence Corollary \ref{c:1} implies that \begin{displaymath}
		\begin{split}
			&\frac{1}{r^2}\mu(\{c_o^F(\psi)\psi-\log|F|<\log r\})\\
			\ge& \sup_{\delta\in\mathbb{Z}_{>0}}\frac{C_{F^{1+\delta},2c_o^F(\psi)\psi+\delta\max\{2c_o^{F}(\psi)\psi,2\log|F|\}}(o)}{\left(1+\frac{1}{\delta}\right)\sup_{D}e^{(1+\delta)\max\{2c_o^{F}(\psi)\psi,2\log|F|\}}}.
		\end{split}
	\end{displaymath}
	holds for any $r\in(0,1]$. We prove the remark in Section \ref{sec:p-r}.
\end{Remark}

When $F\equiv1$, Theorem \ref{thm:J_M} is the following sharp effectiveness result related to Conjecture D-K.

\begin{Corollary}[\cite{G16}]
	If $c_o^{f}(\psi)<+\infty$, then
	\begin{displaymath}
		\frac{1}{r^2}\int_{\{c_o^{f}(\psi)\psi<\log r\}}|f|^2\ge G(0;\Psi_1,I_+(\Psi_1)_o,f)>0
	\end{displaymath}
	holds for any $r\in(0,1]$, where $\Psi_1=\min\{2c_o^f(\psi)\psi,0\}$.
\end{Corollary}

\subsubsection{The strong openness property of $I(a\Psi)_o$}

\

In this section, we present a strong openness property of the module $I(a\Psi)_o$.

Let $F$ be a holomorphic function on a pseudoconvex domain $D\subset\mathbb{C}^n$ containing the origin $o\in\mathbb{C}^n$, and let $\psi$ be a plurisubharmonic function on $D$. Denote
$$\Psi:=\min\{\psi-2\log|F|,0\}.$$
Let $f$ be a holomorphic function on $\{\Psi<-t_0\}$ such that $f_o\in I_o$. Set
\[a_o^f(\Psi):=\sup\big\{a\ge0:f_o\in I(2a\Psi)_o\big\}.\]
Especially, when $F\equiv1$ and $\psi(o)=-\infty$, $a_o^f(\Psi)$ is the jumping number $c_o^f(\psi)$.

Using Theorem \ref{thm:main}, we obtain the following estimate for $L^2$ integrals on the sublevel sets of $\Psi$:

\begin{Theorem}
	\label{thm:J-M2}Assume that $a_o^f(\Psi)<+\infty$, then we have $a_o^f(\Psi)>0$ and
	\begin{displaymath}
		\frac{1}{r^2}\int_{\{a_o^f(\Psi)\Psi<\log r\}}|f|^2\ge G(0;\Psi,I_+(2a_o^f(\Psi)\Psi)_o,f)>0
	\end{displaymath}
	holds for any $r\in(0,e^{-a_o^f(\Psi)t_0}]$.
\end{Theorem}

Theorem \ref{thm:J-M2} implies the following strong openness property of $I(a\Psi)_o$.
\begin{Corollary}
	\label{c:3}
	$I(a\Psi)_o=I_+(a\Psi)_o$ holds for any $a\ge0$.
\end{Corollary}

When $F\equiv1$ and $\psi(o)=-\infty$, Corollary \ref{c:3} is the strong openness property of multiplier ideal sheaves (\cite{GZSOC}).

\subsubsection{A lower semicontinuity property}
\

Let $\{\psi_m\}_{m\in\mathbb{Z}_{>0}}$ be a sequence of  plurisubharmonic functions on a pseudoconvex domain $\Delta^n$, 
and let $\{F_m\}_{m\in\mathbb{Z}_{>0}}$ be a sequence of holomorphic functions on $\Delta^n$. 
Denote $\Psi_m:=\min\{\psi_m-2\log|F_m|,0\}$. Assume that $\{\Psi_m\}_{m\in\mathbb{Z}_{>0}}$ converges to a Lebesgue measurable function $\Psi$ on $\Delta^n$ in Lebesgue measure.
Let $\{f_m\}_{m\in\mathbb{Z}_{>0}}$ be a sequence of holomorphic functions on $\{\Psi_m<-t_m\}$, where $t_m>0$. 

In this section, we present a lower semicontinuity property of $\Psi_m$ with respect to $I(\Psi_m)_o$.
\begin{Proposition}
	\label{p:semi}
 Assume that there exist $t_0\ge t_m$ for any $m\in\mathbb{Z}_{>0}$ and a sequence of Lebesgue functions $\{\tilde f_m\}_{\mathbb{Z}_{>0}}$ on $\Delta^n$ with uniform bound satisfying that $\tilde f_m=f_m$ on $\{\Psi_m<-t_0\}$ for any $m\in\mathbb{Z}_{>0}$ and $\{\tilde f_m\}_{\mathbb{Z}_{>0}}$ converges to a Lebesgue measurable function $\tilde f$ on $\Delta^n$ in Lebesgue measure. Assume that for any pseudoconvex domain $D\subset\Delta^n$ containing the origin $o$,
	$$\inf_mG(0;\Psi_m,I(\Psi_m)_o,f_m)>0.$$
Then we have
	$$|\tilde f|^2e^{-\Psi}\not\in L^1(U),$$
	where $U$ is any neighborhood  of  $o$.
	
If $\{\psi_m\}_{m\in\mathbb{Z}_{>0}}$ and $\{\log|F_m|\}_{m\in\mathbb{Z}_{>0}}$ converge to a plurisubharmonic function $\psi$ and $\log|F|$ in Lebesgue measure respectively, where $F$ is a holomorphic function, and there exists a holomorphic function $f$ on $\{\psi-2\log|F|<-t\}\cap V$ such that $f=\tilde f$ on $\{\psi-2\log|F|<-t\}\cap V$, where $t>0$ and $V$ is a neighborhood of $o$, we have
	$$f_o\not\in{I}(\Psi)_o.$$	
\end{Proposition}

When $F_m\equiv1$ and $\psi_m(o)=-\infty$ for any $m\in\mathbb{Z}_{>0}$, Proposition \ref{p:semi} is Proposition 1.8 in \cite{GZeff}.

\vspace{.1in} {\em Acknowledgements}. The second named author was supported by National Key R\&D Program of China 2021YFA1003100, NSFC-11825101, NSFC-11522101 and NSFC-11431013. The third named author was supported by China Postdoctoral Science Foundation BX20230402 and 2023M743719.

\section{Preparations}

\subsection{$L^2$ methods}\label{sec:L2.methods}
\

Let $F$ be a holomorphic function on a pseudoconvex domain $D\subset\mathbb{C}^n$, and let $\psi$ be a plurisubharmonic function on $D$.  Let $\delta$ be a positive integer. Denote
$$\varphi:=(1+\delta)\max\{\psi,2\log|F|\},$$
 and
 $$\Psi:=\min\{\psi-2\log|F|,0\}.$$
 If $F(z)=0$ for $z\in D$, we set $\Psi(z)=0$. Then $\varphi+\Psi$ and $\varphi+(1+\delta)\Psi$ are both plurisubharmonic functions on $D$.

\begin{Lemma}
	\label{l:L2}
	Let $B\in(0,+\infty)$ and $t_0\in(0,+\infty)$ be arbitrarily given. Let $f$ be a holomorphic function on $\{\Psi<-t_0\}$ such that
	\begin{displaymath}
		\int_{\{\Psi<-t_0\}}|f|^2<+\infty.
	\end{displaymath}
	Then there exists a holomorphic function $\tilde F$ on $D$ such that
	\begin{displaymath}\begin{split}
		&\int_{D}|\tilde F-(1-b_{t_0,B}(\Psi))fF^{1+\delta}|^2e^{-\varphi+v_{t_0,B}(\Psi)-\Psi}\\
		\le&\left(\frac{1}{\delta}+1-e^{-t_0-B} \right)\int_D\frac{1}{B}\mathbb{I}_{\{-t_0-B<\Psi<-t_0\}}|f|^2e^{-\Psi},		
	\end{split}	\end{displaymath}
	where $b_{t_0,B}(t)=\int_{-\infty}^t\frac{1}{B}\mathbb{I}_{\{-t_0-B<s<-t_0\}}ds$ and $v_{t_0,B}(t)=\int_0^tb_{t_0,B}(s)ds$.
\end{Lemma}
We give the proof of Lemma \ref{l:L2} in Appendix. Lemma \ref{l:L2} implies the following lemma.

\begin{Lemma}[see \cite{G16}, see also \cite{GY-support,GY-concavity}]
	\label{l:L2'}Let $B\in(0,+\infty)$ and $t_0>t_1\ge0$ be arbitrarily given. Let $f$ be a holomorphic function on $\{\Psi<-t_0\}$ such that
	\begin{displaymath}
		\int_{\{\Psi<-t_0\}}|f|^2<+\infty.
	\end{displaymath}
	Then there exists a holomorphic function $\tilde F$ on $
	\{\Psi<-t_1\}$ such that
	\begin{displaymath}\begin{split}
		&\int_{\{\Psi<-t_1\}}|\tilde F-(1-b_{t_0,B}(\Psi))f|^2e^{v_{t_0,B}(\Psi)-\Psi}\\
		\le&\left(e^{-t_1}-e^{-t_0-B} \right)\int_D\frac{1}{B}\mathbb{I}_{\{-t_0-B<\Psi<-t_0\}}|f|^2e^{-\Psi},		
	\end{split}	\end{displaymath}
	where $b_{t_0,B}(t)=\int_{-\infty}^t\frac{1}{B}\mathbb{I}_{\{-t_0-B<s<-t_0\}}ds$ and $v_{t_0,B}(t)=\int_0^tb_{t_0,B}(s)ds$.
\end{Lemma}
\begin{proof}
	Let $\delta$ be any positive integer. Set $\tilde\varphi:=(1+\delta)\max\{\psi+t_1,2\log|F|\}$ and $\tilde\Psi:=\{\psi+t_1-2\log|F|,0\}$. Note that $\{\Psi<-t_0\}=\{\tilde\Psi<-(t_0-t_1)\}$ and $\tilde\Psi=\Psi+t_1$ on $\{\Psi<-t_1\}$. Lemma \ref{l:L2} shows that there exists a holomorphic function $\tilde F_{\delta}$ on $D$ such that
		\begin{displaymath}\begin{split}
		&\int_{D}|\tilde F_{\delta}-(1-b_{t_0-t_1,B}(\tilde\Psi))fF^{1+\delta}|^2e^{-\tilde\varphi+v_{t_0-t_1,B}(\tilde\Psi)-\tilde\Psi}\\
		\le&\left(\frac{1}{\delta}+1-e^{-t_0-B+t_1} \right)\int_D\frac{1}{B}\mathbb{I}_{\{-t_0-B+t_1<\tilde\Psi<-t_0+t_1\}}|f|^2e^{-\tilde\Psi},		
	\end{split}	\end{displaymath}
	which implies that
	\begin{equation}
		\label{eq:0211c}\begin{split}
		&\int_{\{\Psi<-t_1\}}|\tilde F_{\delta}-(1-b_{t_0,B}(\Psi))fF^{1+\delta}|^2e^{-\varphi+v_{t_0,B}(\Psi)-\Psi}\\
		=&\int_{\{\Psi<-t_1\}}|\tilde F_{\delta}-(1-b_{t_0-t_1,B}(\tilde\Psi))fF^{1+\delta}|^2e^{-\tilde\varphi+v_{t_0-t_1,B}(\tilde\Psi)-\tilde\Psi}\\
		\le&\int_{D}|\tilde F_{\delta}-(1-b_{t_0-t_1,B}(\tilde\Psi))fF^{1+\delta}|^2e^{-\tilde\varphi+v_{t_0-t_1,B}(\tilde\Psi)-\tilde\Psi}\\
		\le&\left(\frac{1}{\delta}e^{-t_1}+e^{-t_1}-e^{-t_0-B} \right)\int_D\frac{1}{B}\mathbb{I}_{\{-t_0-B<\Psi<-t_0\}}|f|^2e^{-\Psi}.
		\end{split}
	\end{equation}
	Let $F_{\delta}=\frac{\tilde F_{\delta}}{F^{1+\delta}}$ be a holomorphic function on $\{\Psi<-t_1\}$. Since $|F|^{2(1+\delta)}e^{-\varphi}=1$ on $\{\Psi<-t_1\}$, inequality \eqref{eq:0211c} becomes
	\begin{equation}
		\label{eq:0211d}\begin{split}
				&\int_{\{\Psi<-t_1\}}|F_{\delta}-(1-b_{t_0,B}(\Psi))f|^2e^{v_{t_0,B}(\Psi)-\Psi}\\
		\le&\left(\frac{1}{\delta}e^{-t_1}+e^{-t_1}-e^{-t_0-B} \right)\int_D\frac{1}{B}\mathbb{I}_{\{-t_0-B<\Psi<-t_0\}}|f|^2e^{-\Psi}.
		\end{split}
	\end{equation}
	As $\int_{\{\Psi<-t_0\}}|f|^2<+\infty$ and $v_{t_0,B}(\Psi)-\Psi\ge0$, it follows from inequality \eqref{eq:0211d}  that
	\begin{displaymath}
		\begin{split}
			&\sup_{\delta}\int_{\{\Psi<-t_1\}}|F_{\delta}|^2\\
			\le&2\int_{\{\Psi<-t_1\}}|(1-b_{t_0,B}(\Psi))f|^2+2\sup_{\delta}\int_{\{\Psi<-t_1\}}|F_{\delta}-(1-b_{t_0,B}(\Psi))f|^2\\
			\le&2\int_{\{\Psi<-t_0\}}|f|^2+2\sup_{\delta}\left(\frac{1}{\delta}e^{-t_1}+e^{-t_1}-e^{-t_0-B} \right)\int_D\frac{1}{B}\mathbb{I}_{\{-t_0-B<\Psi<-t_0\}}|f|^2e^{-\Psi}\\
			<&+\infty,
		\end{split}
	\end{displaymath}
	yielding that there exists a subsequence of $\{F_{\delta}\}$ (also denoted by $\{F_{\delta}\}$) compactly convergent to a holomorphic function $\tilde F$ on $\{\Psi<-t_1\}$.
It follows from Fatou's Lemma and inequality \eqref{eq:0211d} that
\begin{displaymath}
	\begin{split}
		&\int_{\{\Psi<-t_1\}}|F-(1-b_{t_0,B}(\Psi))f|^2e^{v_{t_0,B}(\Psi)-\Psi}\\
		=&\int_{\{\Psi<-t_1\}}\lim_{\delta\rightarrow+\infty}|F_{\delta}-(1-b_{t_0,B}(\Psi))f|^2e^{v_{t_0,B}(\Psi)-\Psi}
		\\
		\le&\liminf_{\delta\rightarrow+\infty}\int_{\{\Psi<-t_1\}}|F_{\delta}-(1-b_{t_0,B}(\Psi))f|^2e^{v_{t_0,B}(\Psi)-\Psi}\\
		\le&\liminf_{\delta\rightarrow+\infty}\left(\frac{1}{\delta}e^{-t_1}+e^{-t_1}-e^{-t_0-B} \right)\int_D\frac{1}{B}\mathbb{I}_{\{-t_0-B<\Psi<-t_0\}}|f|^2e^{-\Psi}\\
		=&\left(e^{-t_1}-e^{-t_0-B} \right)\int_D\frac{1}{B}\mathbb{I}_{\{-t_0-B<\Psi<-t_0\}}|f|^2e^{-\Psi}.
	\end{split}
\end{displaymath}
Thus, Lemma \ref{l:L2'} holds.
\end{proof}

\subsection{$\mathcal{O}_{\mathbb{C}^n,o}$-module $I(a\Psi)_o$}\label{sec:module}
\

Let $F$ be a holomorphic function on a pseudoconvex domain $D\subset\mathbb{C}^n$ containing the origin $o\in\mathbb{C}^n$, and let $\psi$ be a plurisubharmonic function on $D$. Denote $\Psi:=\min\{\psi-2\log|F|,0\}$ and $\varphi:=2\max\{\psi,2\log|F|\}$. In this section, we discuss the $\mathcal{O}_{\mathbb{C}^n,o}$-module $I(a\Psi)_o$, where $a\ge0$.

Let $k$ be a positive integer. Denote $I_o:=I(0\Psi)_o$.

\begin{Lemma}
	\label{l:m1}For any $f_o\in I_o$, there exists a pseudoconvex domain $D_0\subset D$ containing $o$ and a holomorphic function $\tilde F$ on $D_0$ such that $(\tilde F,o)\in\mathcal{I}(k\varphi)_o$ and  $\int_{\{\Psi<-t\}\cap D_0}|\tilde F-fF^{2k}|^2e^{-k\varphi-k\Psi}<+\infty$ for some $t>0$.
\end{Lemma}
\begin{proof}
	For any $f_o\in I_o$, there exists a pseudoconvex domain $D_0\subset D$ containing $o$ and $t_0>0$ such that $f\in\mathcal{O}(\{k\Psi<-t_0\}\cap D_0)$ and $\int_{\{k\Psi<-t_0\}\cap D_0}|f|^2<+\infty$.  It follows from Lemma \ref{l:L2} that there exists a holomorphic function $\tilde F$ on $D_0$ such that
	\begin{equation}
		\label{eq:0220a}
		\begin{split}
			&\int_{D_0}|\tilde F-(1-b(k\Psi))fF^{2k}|^2e^{-k\varphi+v(k\Psi)-k\Psi}\\
			\le&(2-e^{-t_0-1})\int_{D_0}\mathbb{I}_{\{-t_0-1<k\Psi<-t_0\}}|f|^2e^{-k\Psi},
		\end{split}
	\end{equation}
	where $b(t)=\int_{-\infty}^t\mathbb{I}_{\{-t_0-1<s<-t_0\}}ds$ and $v(t)=\int_0^tb(s)ds$. Note that $v(t)\ge-t_0-1$ on $\mathbb{R}$ and
 $b(t)=0$ on $(-\infty,-t_0-1)$. Inequality \eqref{eq:0220a} implies that
	\begin{displaymath}
		\begin{split}
			&\int_{D_0\cap\{k\Psi<-t_0-1\}}|\tilde F-fF^{2k}|^2e^{-k\varphi-k\Psi}\\
			\le & e^{t_0+1}\int_{D_0}|\tilde F-(1-b(k\Psi))fF^{2k}|^2e^{-k\varphi+v(k\Psi)-k\Psi}\\
			\le & (2e^{t_0+1}-1)\int_{D_0}\mathbb{I}_{\{-t_0-1<k\Psi<-t_0\}}|f|^2e^{-k\Psi}\\
			< & +\infty.
		\end{split}
	\end{displaymath}
	Consequently, since $v(k\Psi)-k\Psi\ge0$, $b(t)=1$ on $[-t_0,+\infty)$, and $|F|^{4k}e^{-k\varphi}=1$ on $\{k\Psi<-t_0\}$, we deduce from Inequality \eqref{eq:0220a} that
\begin{displaymath}\begin{split}
		&\int_{D_0}|\tilde F|^2e^{-k\varphi}\\
		\le &2\int_{D_0}|(1-b(k\Psi))fF^{2k}|^2e^{-k\varphi}+2\int_{D_0}|\tilde F-(1-b(k\Psi))fF^{2k}|^2e^{-k\varphi}\\
		\le &2\int_{D_0\cap\{k\Psi<-t_0\}}|f|^2+2(2-e^{-t_0-1})\int_{D_0}\mathbb{I}_{\{-t_0-1<k\Psi<-t_0\}}|f|^2e^{-k\Psi}\\
		<&+\infty.
		\end{split}
	\end{displaymath}
Thus, Lemma \ref{l:m1} holds.
\end{proof}

 Let $f_o\in I_o$. Taking any  $(\tilde F,o)\in \mathcal{I}(k\varphi)_o$ and $(\tilde F_1,o)\in \mathcal{I}(k\varphi)_o$, 	if there exist $t_1>0$ and a neighborhood $D_1$ of $o$ such that $\int_{\{\Psi<-t_1\}\cap D_1}|\tilde F-fF^{2k}|^2e^{-k\varphi-k\Psi}<+\infty$ and $\int_{\{\Psi<-t_1\}\cap D_1}|\tilde F_1-fF^{2k}|^2e^{-k\varphi-k\Psi}<+\infty$, then we have
	\begin{equation}
		\label{eq:0220c}\int_{\{\Psi<-t_1\}\cap D_1}|\tilde F_1-\tilde F|^2e^{-k\varphi-k\Psi}<+\infty,
	\end{equation}
	and there exists a neighborhood $D_2$ of $o$ such that
	\begin{equation}
		\label{eq:0220d}
		\int_{D_2}|\tilde F-\tilde F_1|^2e^{-k\varphi}<+\infty.
	\end{equation}
	Combining inequality \eqref{eq:0220c} and inequality \eqref{eq:0220d}, we obtain that $(\tilde F-\tilde F_1,o)\in\mathcal{I}(k\varphi+k\Psi)_o$. Thus, according to Lemma \ref{l:m1}, there exists
	  a map
	  \[\tilde P:I_o\rightarrow \mathcal{I}(k\varphi)_o/\mathcal{I}(k\varphi+k\Psi)_o,\]
	  given by
	  $$\tilde P(f_o)=[(\tilde F,o)],$$
	  for any $f_o\in I_o$, where $(\tilde F,o)\in \mathcal{I}(k\varphi)_o$ such that  $\int_{\{\Psi<-t\}\cap D_0}|\tilde F-fF^{2k}|^2e^{-k\varphi-k\Psi}<+\infty$ for some $t>0$ and some neighborhood $D_0$ of $o$, and $[(\tilde F,o)]$ is the equivalence class of $(\tilde F,o)$ in $\mathcal{I}(k\varphi)_o/\mathcal{I}(k\varphi+k\Psi)_o$.
\begin{Lemma}\label{l:m2}
	$\tilde P$ is an  $\mathcal{O}_{\mathbb{C}^n,o}$-module homomorphism, and $\mathrm{Ker}(\tilde P)=I(k\Psi)_o$.
\end{Lemma}
	 \begin{proof}
	 	For any  $f_o\in I_o$ and $\tilde f_o\in I_o$, take $[(\tilde F,o)]=\tilde P(f_o)$ and $[(\tilde F_1,o)]=\tilde P(\tilde f_o)$. Then we have $(\tilde F+\tilde F_1,o)\in\mathcal{I}(k\varphi+k\Psi)_o$, and
		\begin{flalign*}
			\begin{split}
				&\int_{\{\Psi<-t_1\}\cap D_1}|\tilde F_1+\tilde F-(f+\tilde f)F^{2k}|^2e^{-k\varphi-k\Psi}\\
				\le & 2\int_{\{\Psi<-t_1\}\cap D_1}|\tilde F-fF^{2k}|^2e^{-k\varphi-k\Psi}+2\int_{\{\Psi<-t_1\}\cap D_1}|\tilde F_1-\tilde fF^{2k}|^2e^{-k\varphi-k\Psi}\\
				<&+\infty,
			\end{split}
		\end{flalign*}
		for some $t_1>0$ and some neighborhood $D_1$ of $o$, which shows that
		\[\tilde P(f_o+\tilde f_o)=\tilde P((f+\tilde f)_o)=[(\tilde F+\tilde F_1,o)]=\tilde P(f_o)+\tilde P(\tilde f_o).\]

		For  any $(h,o)\in\mathcal{O}_{\mathbb{C}^n,o}$, we have $(h\tilde F,o)\in\mathcal{I}(k\varphi+k\Psi)_o$ and
		\[\int_{\{\Psi<-t_2\}\cap D_2}|h\tilde F-hfF^{2k}|^2e^{-k\varphi-k\Psi}\le C\int_{\{\Psi<-t_2\}\cap D_2}|\tilde F-fF^{2k}|^2e^{-k\varphi-k\Psi}<+\infty,\]
		for some $C>0$, $t_2>0$ and a neighborhood $D_2$ of $o$, which shows that
		\[\tilde P((h,o)\cdot f_o)=\tilde P((hf)_o)=[(h\tilde F,o)]=(h,o)\cdot\tilde P(f_o).\]

		Thus, $\tilde P$ is an  $\mathcal{O}_{\mathbb{C}^n,o}$-module homomorphism from $I_o$ to $\mathcal{I}(k\varphi)_o/\mathcal{I}(k\varphi+k\Psi)_o$.
	
	Now, we prove $\mathrm{Ker}(\tilde P)=I(k\Psi)_o$. For any $f_o\in I(k\Psi)_o$, $\tilde P(f_o)=0$ holds if and only if  there exist a neighborhood $D_3$ of $o$ and $t_3>0$ such that
	$$\int_{\{\Psi<-t_3\}\cap D_3}|f_oF^{2k}|^2e^{-k\varphi-k\Psi}=\int_{\{\Psi<-t_3\}\cap D_3}|f_o|^2e^{-k\Psi}<+\infty,$$ i.e. $f_o\in I(k\Psi)_o$.
	 \end{proof}
	
	Following from  Lemma \ref{l:m2}, there exists an $\mathcal{O}_{\mathbb{C}^n,o}$-module homomorphism
	\[P:I_o/I(k\Psi)_o\rightarrow \mathcal{I}(k\varphi)_o/\mathcal{I}(k\varphi+k\Psi)_o,\]
	given by
	  $$P([f_o]_1)=\tilde P(f_o),$$
	for any $f_o\in I_o$, where $[f_o]_1$ is the equivalence class of $f_o$ in $I_o/I(k\Psi)_o$.
	
	 Note that $|F^{2k}|^2e^{-k\varphi}=1$ on $\{\Psi<0\}$. For any $(\tilde F,o)\in\mathcal{I}(k\varphi)_o$, it is clear that $\left(\frac{\tilde F}{F^{2k}}\right)_o\in I_o$. Moreover, if $(\tilde F,o)\in\mathcal{I}(k\varphi+k\Psi)_o$, we have $\left(\frac{\tilde F}{F^{2k}}\right)_o\in I(k\Psi)_o$. Hence, there exists an $\mathcal{O}_{\mathbb{C}^n,o}$-module homomorphism
	 \[Q:\mathcal{I}(k\varphi)_o/\mathcal{I}(k\varphi+k\Psi)_o\rightarrow I_o/I(k\Psi)_o,\]
	 given by
	  $$Q([(\tilde F,o)])=\left[\left(\frac{\tilde F}{F^{2k}}\right)_o\right]_1,$$
	  for any $[(\tilde F,o)]\in\mathcal{I}(k\varphi)_o/\mathcal{I}(k\varphi+k\Psi)_o$.

\begin{Lemma}
	\label{l:m3}
	$P$ is an  $\mathcal{O}_{\mathbb{C}^n,o}$-module isomorphism from $I_o/I(k\Psi)_o$ to $\mathcal{I}(k\varphi)_o/\mathcal{I}(k\varphi+k\Psi)_o$ and $P^{-1}=Q$.
\end{Lemma}
	\begin{proof}
		According to the definitions of $Q$ and $P$, we have that $$P\circ Q([(\tilde F,o)])=[(\tilde F,o)]$$
		 for any $[(\tilde F,o)]\in\mathcal{I}(k\varphi)_o/\mathcal{I}(k\varphi+k\Psi)_o$, which implies  that $P$ is surjective. Lemma \ref{l:m2} shows that $P$ is injective. Thus, we get that $P$ is an  $\mathcal{O}_{\mathbb{C}^n,o}$-module isomorphism and $P^{-1}=Q$.
	\end{proof}
Let $a\in[0,k)$. Denote $P_a:=P|_{I(a\Psi)_o/I(k\Psi)_o}$, where $I(a\Psi)_o/I(k\Psi)_o$ is an $\mathcal{O}_{\mathbb{C}^n,o}$-submodule of $I_o/I(k\Psi)_o$. It is clear that $k\varphi+a\Psi$ is a plurisubharmonic function on $D$. The the following lemma gives an isomorphism between $I(a\Psi)_o/I(k\Psi)_o$ and $\mathcal{I}(k\varphi+a\Psi)_o/\mathcal{I}(k\varphi+k\Psi)_o$.
\begin{Lemma}
	\label{l:m4}
	$P_a$ is an  $\mathcal{O}_{\mathbb{C}^n,o}$-module isomorphism from $I(a\Psi)_o/I(k\Psi)_o$ to $\mathcal{I}(k\varphi+a\Psi)_o/\mathcal{I}(k\varphi+k\Psi)_o$.
\end{Lemma}	
\begin{proof}
	It suffices to prove $\mathrm{Im}(P_a)=\mathcal{I}(k\varphi+a\Psi)_o/\mathcal{I}(k\varphi+k\Psi)_o$. For any $f_o\in I(a\Psi)_o$, taking $[(\tilde F,o)]=P_a([f_o]_1)=P([f_o]_1)$, we have that
	 there exist $t>0$ and a neighborhood $D_0$ of $o$ such that
	\begin{equation}
\label{eq:0221a}
		\begin{split}
			&\int_{\{\Psi<-t\}\cap D_0}|\tilde F|^2e^{-k\varphi-a\Psi}\\
			\le&2\int_{\{\psi<-t\}\cap D_0}|f|^2e^{-a\Psi}+2\int_{\{\Psi<-t\}\cap D_0}|\tilde F-fF^{2k}|^2e^{-k\varphi-a\Psi}\\
			<&+\infty.
		\end{split}
	\end{equation}
	Combining $(\tilde F,o)\in\mathcal{I}(k\varphi)_o$ and inequality \eqref{eq:0221a}, we get $(\tilde F,o)\in\mathcal{I}(k\varphi+a\Psi)_o$, which shows $\mathrm{Im}(P_a)\subset\mathcal{I}(k\varphi+a\Psi)_o/\mathcal{I}(k\varphi+k\Psi)_o$. For any $(\tilde F_1,o)\in\mathcal{I}(k\varphi+a\Psi)_o$, it is clear that $\left(\frac{\tilde F_1}{F^{2k}}\right)_o\in I(a\Psi)_o$. Then we obtain
	\[P^{-1}([(\tilde F_1,o)])=Q([(\tilde F_1,o)])\in I(a\Psi)_o/I(k\Psi)_o,\]
	which implies that $\mathrm{Im}(P_a)=\mathcal{I}(k\varphi+a\Psi)_o/\mathcal{I}(k\varphi+k\Psi)_o$.
\end{proof}

According to the definition of $I(a\Psi)_o$, we have that $I(a\Psi)_o\subset I(a'\Psi)_o$ for any $0\le a'<a<+\infty$. Denote
\[I_+(a\Psi)_o:=\bigcup_{p>a}I(p\Psi)_o,\]
which is an $\mathcal{O}_{\mathbb{C}^n,o}$-submodule of $I_o$, where $a\ge0$.
\begin{Lemma}
	\label{l:m5} There exists $a'>a$ such that $I(a'\Psi)_o=I_+(a\Psi)_o$ for any $a\ge0$.
\end{Lemma}
\begin{proof}
By the definition of $I_+(a\Psi)_o$, $I(p\Psi)_o\subset I_+(a\Psi)_o$ for any $p>a$. Then it suffices to prove that  there exists $a'>a$ such that $I_+(a\Psi)_o\subset I(a'\Psi)_o$.

 Let $k>a$ be an integer. It is clear that $\mathcal{I}(k\varphi+a\Psi)_o \subset \mathcal{I}(k\varphi+a'\Psi)$ for any $0\le a'<a\le k$. Denote
	\[\mathfrak{I}:=\bigcup_{a<p<k}\mathcal{I}(k\varphi+p\Psi)_o,\]
	 which is an ideal of $\mathcal{O}_{\mathbb{C}^n,o}$. It follows from Lemma \ref{l:m3} that $P|_{I_+(a\Psi)_o/I(k\Psi)_o}$ is an  $\mathcal{O}_{\mathbb{C}^n,o}$-module isomorphism from $I_+(a\Psi)_o/I(k\Psi)_o$ to $\mathfrak{I}/\mathcal{I}(k\varphi+k\Psi)_o$. As $\mathcal{O}_{\mathbb{C}^n,o}$ is a Noetherian ring (see \cite{hormander}), we get that $\mathfrak{I}$ is finitely generated. Hence, there exists a finite set $\{(f_1)_o,\ldots,(f_m)_o\}\subset I_+(a\Psi)_o$, which satisfies that for any $f_o\in I_+(a\Psi)_o$, there exists $(h_j,o)\in\mathcal{O}_{\mathbb{C}^n,o}$ for any $1\le j\le m$, such that
	 $$f_o-\sum_{j=1}^{m}(h_j,o)\cdot (f_j)_o\in I(k\Psi)_o.$$
	  By the definition of $I_+(a\Psi)_o$, there exists $a'\in(a,k)$ such that $\{(f_1)_o,\ldots,(f_m)_o\}\subset I(a'\Psi)_o$. As $(h_j,o)\cdot (f_j)_o\in I(a'\Psi)_o$ for any $1\le j\le m$ and $I(k\Psi)_o\subset I(a'\Psi)_o$, we obtain that $I_+(a\Psi)_o\subset I(a'\Psi)_o$.
	
	  Thus, Lemma \ref{l:m5} holds.
\end{proof}

We recall the closedness of submodules, which can be referred to  \cite{G-R} and will be used in the proof of Lemma \ref{l:converge}.

\begin{Lemma}[see \cite{G-R}]
\label{closedness}
Let $N$ be a submodule of $\mathcal O_{\mathbb C^n,o}^q$, $1\leq q<+ \infty$, let $f_j\in\mathcal O_{\mathbb C^n,o}(U)^q$ be a sequence of $q-$tuples holomorphic in an open neighborhood $U$ of the origin $o$. Assume that the $f_j$ converge uniformly in $U$ towards to a $q-$tuples $f\in\mathcal{O}_{\mathbb C^n,o}(U)^q$, assume furthermore that all germs $(f_{j},o)$ belong to $N$. Then $(f,o)\in N$.	
\end{Lemma}

The following lemma will be used in the proof of Theorem \ref{thm:main}.

\begin{Lemma}
	\label{l:converge}Let $J$ be an $\mathcal{O}_{\mathbb{C}^n,o}$-submodule of $I_o$ such that $I(k\Psi)_o\subset J$ for some  positive integer $k$, and let $f\in J(\Psi)_o$. Let $D_0\subset D$ containing $o$ be a pseudoconvex domain, and let $f_j$ be a sequence of holomorphic functions on $\{\Psi<-s_j\}\cap D_0$ for any $j\in\mathbb{Z}_{>0}$, where $s_j\ge0$. Assume that $s_0:=\lim_{j\rightarrow+\infty}s_j\in[0,+\infty)$,
	$$\limsup_{j\rightarrow+\infty}\int_{\{\Psi<-s_j\}\cap D_0}|f_j|^2\le C<+\infty,$$
	and $(f_j-f)_o\in J$. Then there exists a subsequence of $\{f_j\}_{j\in\mathbb{Z}_{>0}}$  compactly convergent to a holomorphic function $f_0$ on $\{\Psi<-s_0\}\cap D_0$, which satisfies that
	$$\int_{\{\Psi<-s_0\}\cap D_0}|f_0|^2\le C,$$
	and $(f_0-f)_o\in J$.
\end{Lemma}
\begin{proof}
	Since $\limsup\limits_{j\rightarrow+\infty}\int_{\{\Psi<-s_j\}\cap D_0}|f_j|^2<+\infty$ and $\lim\limits_{j\rightarrow+\infty}s_j=s_0$, we can extract a subsequence of $\{f_j\}_{j\in\mathbb{Z}_{>0}}$ (also denoted by $\{f_j\}_{j\in\mathbb{Z}_{>0}}$) compactly convergent to a holomorphic function $f_0$ on $\{\Psi<-s_0\}\cap D_0$ and
	$$\int_{\{\Psi<-s_0\}\cap D_0}|f_0|^2\leq C,$$
	which implies $(f_0)_o\in I_o$.
	Thus, it suffices to prove $(f_0-f)_o\in J$.
	
	Denote $\varphi:=2\max\{\psi,2\log|F|\}$, and $t_j:=ks_j$ for $j\in\mathbb{Z}_{\ge0}$. For any $j\in\mathbb{Z}_{>0}$, it follows from Lemma \ref{l:L2} that there exists a holomorphic function $F_j$ on $D_0$ satisfying that
	\begin{equation}
		\label{eq:0213a}
		\begin{split}
		&\int_{D_0}|F_j-(1-b_{t_j,1}(k\Psi))f_jF^{2k}|^2e^{-k\varphi+v_{t_j,1}(k\Psi)-k\Psi}\\
		\le&(2-e^{-t_j-1})\int_{D_0}\mathbb{I}_{\{-t_j-1<k\Psi<-t_j\}}|f_j|^2e^{-k\Psi}\\
		\le&(2e^{t_j+1}-1)\int_{\{k\Psi<-t_j\}\cap D_0}|f_j|^2,	
	\end{split}	\end{equation}
	where $b_{t_j,1}(t)=\int_{-\infty}^t\mathbb{I}_{\{-t_j-1<s<-t_j\}}ds$ and $v_{t_j,1}(t)=\int_0^tb_{t_j,1}(s)ds$. Note that $b_{t_j,1}(t)=0$ for $t<-t_j-1$. Then we have
	\begin{equation}
		\label{eq:0213b}\int_{\{k\Psi<-t_j-1\}\cap D_0}|F_j-f_jF^{2k}|^2e^{-k\varphi-k\Psi}<+\infty.
	\end{equation}
In addition, note that $b_{t_j,1}(t)=1$ for $t\ge t_j$ and $|F|^{2k}e^{-k\varphi}=1$ on $\{k\Psi<-t_j\}\cap D_0$. Then as $v_{t_j,1}(k\Psi)-k\Psi\ge0$, it follows from inequality \eqref{eq:0213a} that
\begin{equation}
	\label{eq:0213c}
	\begin{split}
		&\int_{D_0}|F_j|^2e^{-k\varphi}\\
		\le&2\int_{D_0}|(1-b_{t_j,1}(k\Psi))f_jF^{2k}|^2e^{-k\varphi}+2\int_{D_0}|F_j-(1-b_{t_j,1}(k\Psi))f_jF^{2k}|^2e^{-k\varphi}\\
		\le&(2e^{t_j+1}+1)\int_{\{k\Psi<-t_j\}\cap D_0}|f_j|^2.
			\end{split}
\end{equation}	
Since $\lim_{j\rightarrow+\infty}t_j=t_0$ and $\limsup_{j\rightarrow+\infty}\int_{\{k\Psi<-t_j\}\cap D_0}|f_j|^2<+\infty$, inequality \eqref{eq:0213c} and Fatou's Lemma  imply that there exists a subsequence of $\{F_j\}_{j\in\mathbb{Z}_{>0}}$ denoted by $\{F_{j_l}\}_{l\in\mathbb{Z}_{>0}}$, which compactly converges to a holomorphic function $F_0$ on $D_0$ with
\begin{equation}
	\label{eq:0222a}
	\int_{D_0}|F_0|^2e^{-k\varphi}\le\liminf_{l\rightarrow+\infty}\int_{D_0}|F_{j_l}|^2e^{-k\varphi}<+\infty.
\end{equation}
As $f_j$ converge to $f_0$, by inequality \eqref{eq:0213a} and Fatou's Lemma, we obtain
\begin{equation*}
	\begin{split}
		&\int_{D_0}|F_0-(1-b_{t_0,1}(k\Psi))f_0F^{2k}|^2e^{-k\varphi+v_{t_j,1}(k\Psi)-k\Psi}\\
		=&\int_{D_0}\liminf_{l\rightarrow+\infty}|F_{j_l}-(1-b_{t_{j_l},1}(k\Psi))f_{j_l}F^{2k}|^2e^{-k\varphi+v_{t_{j_l},1}(k\Psi)-k\Psi}\\
		\le&\liminf_{l\rightarrow+\infty}\int_{D_0}|F_{j_l}-(1-b_{t_{j_l},1}(k\Psi))f_{j_l}F^{2k}|^2e^{-k\varphi+v_{t_{j_l},1}(k\Psi)-k\Psi}\\
		<&+\infty,
	\end{split}
\end{equation*}
yielding that
\begin{equation}
	\label{eq:0213d}
	\int_{\{k\Psi<-t_0-1\}\cap D_0}|F_0-f_0F^{2k}|^2e^{-k\varphi-k\Psi}<+\infty.
\end{equation}

Combining inequality \eqref{eq:0213b}, inequality \eqref{eq:0213c}, inequality \eqref{eq:0222a}, inequality \eqref{eq:0213d} and the definition of $P:I_o/I(k\Psi)_o\rightarrow \mathcal{I}(k\varphi)/\mathcal{I}(k\varphi+k\Psi)$, we get
$$P([(f_j)_o]_1)=[(F_j,o)],$$
for any $j\in\mathbb{Z}_{\ge0}$. $(f_j-f)_o\in J$ for any $j\in\mathbb{Z}_{>0}$ implies $(f_j-f_1)_o\in J$ for any $j\in\mathbb{Z}_{>0}$. Lemma \ref{l:m3} shows that there exists an ideal $\tilde J$ of $\mathcal{O}_{\mathbb{C}^n,o}$ such that
\[\mathcal{I}(k\varphi+k\Psi)_o\subset \tilde J\subset \mathcal{I}(k\varphi)_o,\]
and
\[\tilde J/\mathcal{I}(k\varphi+k\Psi)_o=\mathrm{Im}(P|_{J/I(k\Psi)_o}).\]
Then $(F_j-F_1,o)\in \tilde J$, for any $j\in\mathbb{Z}_{>0}$. As $F_j$ compactly converge to $F_0$ on $D_0$,  using Lemma \ref{closedness}, we obtain $(F_0-F_1,o)\in\tilde J$. Recall that $P$ is an $\mathcal{O}_{\mathbb{C}^n,o}$-module isomorphism, and $\tilde J/\mathcal{I}(k\varphi+k\Psi)_o=\mathrm{Im}(P|_{J/I(k\Psi)_o})$. We deduce $(f_0-f_1)_o\in J$, which verifies $(f_0-f)_o\in J$.

Thus, Lemma \ref{l:converge} holds.
	\end{proof}

Let $f$ be a holomorphic function on $D$. Denote that
$$\Psi_1:=\min\{2c_o^{fF}(\psi)\psi-2\log|F|,0\},$$
and $$I_+(\Psi_1)_o:=\bigcup_{a>1}I(a\Psi_1)_o.$$
\begin{Lemma}
	\label{l:m6}
	$f_o\not\in I_+(\Psi_1)_o$.
\end{Lemma}
\begin{proof}
	We prove Lemma \ref{l:m6} by contradiction: if not, there exists $p_0>1$ such that $f_o\in I(p_0\Psi_1)_o$, which implies that
	\begin{equation}
		\label{eq:0221d}
		\int_{\{\Psi_1<-t\}\cap V}|f|^2|F|^{2p_0}e^{-2p_0c_o^{fF}(\psi)\psi}=\int_{\{\Psi_1<-t\}\cap V}|f|^2e^{-p_0\Psi}<+\infty
	\end{equation}
	holds for some $t>0$ and some neighborhood $V\Subset D$ of $o$. Note that
	$$\sup_{V\cap\{\Psi_1\ge t\}}|f|^2|F|^{2p_0}e^{-2p_0c_o^{fF}(\psi)\psi}<+\infty.$$
	Then inequality \eqref{eq:0221d} indicates
	\begin{equation}
		\label{eq:0221e}\int_{V}|f|^2|F|^{2p_0}e^{-2p_0c_o^{fF}(\psi)\psi}<+\infty.
	\end{equation}
	As $\log|F|$ is a plurisubharmonic function on $D$, there exist a neighborhood $U\subset V$ of $o$ and $p'_0\in(1,p_0)$ such that
	\begin{equation}
		\label{eq:0221f}
		\int_{U}|F|^{2(1-p'_0)\frac{p_0}{p_0-p'_0}}=\int_{U}e^{-2(p'_0-1)\frac{p_0}{p_0-p'_0}\log|F|}<+\infty.
	\end{equation}
	It follows from inequality \eqref{eq:0221e}, inequality \eqref{eq:0221f} and H\"older's inequality that
	\begin{displaymath}
		\begin{split}
			&\int_{U}|fF|^2e^{-2p'_0c_o^{fF}(\psi)\psi}\\
			=&\int_U|f|^2|F|^{2p'_0}e^{-2p'_0c_o^{fF}(\psi)\psi}|F|^{2-2p'_0}\\
			\le&\left(\int_{U}|f|^{2\frac{p_0}{p'_0}}|F|^{2p_0}e^{-2p_0c_o^{fF}(\psi)\psi} \right)^{\frac{p'_0}{p_0}}\times\left(\int_U|F|^{2(1-p'_0)\frac{p_0}{p_0-p'_0}} \right)^{\frac{p_0-p'_0}{p_0}}\\
			<&+\infty,
		\end{split}
	\end{displaymath}
	which contradicts the definition of $c_o^{fF}(\psi)$. Thus, we have $f_o\not\in I_+(\Psi_1)_o$.
\end{proof}

\subsection{Some properties of $G(t)$}
\

Following the notations and assumptions in Section \ref{main},
we present some properties related to $G(t;\Psi,J,f)$ ($G(t)$ for short) in this section.

\begin{Lemma}
\label{l:2}
$f_o\in J$ if and only if  $G(t;\Psi,J,f)=0$.
\end{Lemma}

\begin{proof}
If $f_o\in J$, it follows from the definition of $G(t;\Psi,J,f)$  that $G(t;\Psi,J,f)=0$.

If $G(t;\Psi,J,f)=0$, there exists a sequence of holomorphic functions $\{f_j\}_{j\in\mathbb{Z}_{>0}}$ on $\{\Psi<-t\}$ such that $\lim_{j\rightarrow+\infty}\int_{\{\Psi<-t\}}|f_j|^2=0$ and $(f_j-f)_o\in J$. Then Lemma \ref{l:converge} implies that there exists a holomorphic function $f_0$ on $\{\Psi<-t\}$ such that $\int_{\{\Psi<-t\}}|f_0|^2=0$ and $(f_0-f)_o\in J$, which shows $f\in J$.
\end{proof}

The following lemma shows the existence and uniqueness of the minimal holomorphic function.

 \begin{Lemma}
 	\label{l:3}
Assume that $G(t)<+\infty$. Then there exists a unique holomorphic function $f_t$ on $\{\Psi<-t\}$ satisfying $(f_t-f)_o\in J$ and $\int_{\{\Psi<-t\}}|f_t|^2=G(t)$.
Furthermore, for any holomorphic function $\hat
 	{f}$ on $\{\Psi<-t\}$ satisfying $(\hat f-f)_o\in J$ and $\int_{\{\Psi<-t\}}|\hat{f}|^2<+\infty$, we have the following equality
    \begin{equation}
    	\label{eq:l31}
    	\int_{\{\Psi<-t\}}|f_t|^2+\int_{\{\Psi<-t\}}|\hat{f}-f_t|^2=\int_{\{\Psi<-t\}}|\hat{f}|^2.
    \end{equation}	
 \end{Lemma}

 \begin{proof}
 Firstly, we prove the existence of  $f_t$. As $G(t)<+\infty$, there exist holomorphic functions $\{F_j\}_{j\in \mathbb{Z}_{>0}}$ on $\{\Psi<-t\}$ such that
 $\lim_{j\rightarrow+\infty}\int_{\{\Psi<-t\}}|F_j|^2=G(t)$ and $(F_j-f)_o\in J$. Lemma \ref{l:converge} indicates that there exists a holomorphic function $f_t$ on $\{\Psi<-t\}$ such that $(f_t-f)_o\in J$ and $\int_{\{\Psi<-t\}}|f_t|^2\le G(t)$. By the definition of $G(t)$, we have $\int_{\{\Psi<-t\}}|f_t|^2= G(t)$.

 Secondly, we prove the uniqueness of $f_t$ by contradiction: if not, there exist two different holomorphic functions $\tilde f_1$ and $\tilde f_2$ on $\{\Psi<-t\}$ satisfying $\int_{\{\Psi<-t\}}|\tilde f_1|^2=\int_{\{\Psi<-t\}}|\tilde f_2|^2=G(t)$, $(\tilde f_1-f)_o\in J$, and $(\tilde f_2-f)_o\in J$. Note that
\begin{equation}
\label{eq:l32}
\begin{split}	\int_{\{\Psi<-t\}}\left\vert\frac{\tilde f_1+\tilde f_2}{2}\right\vert^2+\int_{\{\Psi<-t\}}\left\vert\frac{\tilde f_1-\tilde f_2}{2}\right\vert^2
	\\=\frac{\int_{\{\Psi<-t\}}|\tilde f_1|^2+\int_{\{\Psi<-t\}}|\tilde f_2|^2}{2}=G(t).
\end{split}
\end{equation}
Then we obtain that
\begin{displaymath}
	\int_{\{\psi<-t\}}\left\vert\frac{f_1+f_2}{2}\right\vert^2< G(t) ,
\end{displaymath}
and $(\frac{\tilde f_1+\tilde f_2}{2}-f)_o\in J$, which contradicts the definition of $G(t)$.

Finally, we prove  equality \eqref{eq:l31}. For any holomorphic function $h$ on $\{\Psi<-t\}$ satisfying $h_o\in J$ and $\int_{\{\Psi<-t\}}|h|^2<+\infty$,
 it is clear that for any complex number $\alpha$, $f_t+\alpha h$ satisfies
$(f_t+\alpha h-f)_o\in J$. The definition of $G(t)$ shows that
\begin{displaymath} \int_{\{\Psi<-t\}}|f_t|^2\leq\int_{\{\Psi<-t\}}|f_t+\alpha h|^2<+\infty,
\end{displaymath}
yielding that
\begin{displaymath}
	\mathrm{Re}\int_{\{\psi<-t\}}f_t\overline{h}=0.
\end{displaymath}
Then we get
\begin{displaymath}
\int_{\{\Psi<-t\}}|f_t+h|^2=\int_{\{\Psi<-t\}}|f_t|^2+\int_{\{\Psi<-t\}}|h|^2.
\end{displaymath}
Choosing $h=\hat{f}-f_t$, we obtain equality \eqref{eq:l31}.\end{proof}

We present the monotonicity and the lower semicontinuity of $G(t)$.

 \begin{Lemma}
 	\label{l:4}
 	$G(t)$ is decreasing on $[0,+\infty)$, such that $\lim_{t\rightarrow t_0+0}G(t)=G(t_0)$ for $t_0\in[0,+\infty)$. If there exists $t_1\ge0$ such that $G(t_1)<+\infty$, then $\lim_{t\rightarrow+\infty}G(t)=0$. Especially, $G(t)$ is lower semicontinuous on $[0,+\infty)$ if $G(0)<+\infty$.
 	 	\end{Lemma}
 	
\begin{proof}
By the definition of $G(t)$, it is clear that $G(t)$ is decreasing on $[0,+\infty)$. The dominated convergence theorem verifies that $\lim_{t\rightarrow+\infty}G(t)=0$ if there exists $t_1\ge0$ such that $G(t_1)<+\infty$. 

It suffices to prove $\lim_{t\rightarrow t_0+0}G(t)= G(t_0)$. We prove it by contradiction: if not, then $\lim_{t\rightarrow t_0+0}G(t)< G(t_0)$ for some $t_0\in[0,+\infty)$. By Lemma \ref{l:3}, there exists a unique holomorphic function $f_t$ on $\{\Psi<-t\}$ satisfying that $(f_t-f)_o\in J$ and $\int_{\{\Psi<-t\}}|f_t|^2=G(t)$. Note that $G_{F,\psi,c}(t)$ is decreasing, indicating that $\lim_{t\rightarrow t_0+0}\int_{\{\Psi<-t\}}|f_t|^2<G(t_0)$. It follows from Lemma \ref{l:converge} that there exists a holomorphic function $\tilde f_{t_0}$ on $\{\Psi<-t_0\}$ such that
$(\tilde f_{t_0}-f)_o\in J$ and $\int_{\{\Psi<-t\}}|\tilde f_{t_0}|^2\le\lim_{t\rightarrow t_0+0}\int_{\{\Psi<-t\}}|f_t|^2<G(t_0)$,
 which contradicts the definition of $G(t_0)$. Thus, we have $\lim_{t\rightarrow t_0+0}G(t)=G(t_0).$
  \end{proof}

We consider the derivatives of $G(t)$ in the following lemma.

\begin{Lemma}
\label{l:5}
Assume that $G(t_1)<+\infty$, where $t_1\in[0,+\infty)$. Then for any $t_0>t_1$, we have
$$\frac{G(t_1)-G(t_0)}{e^{-t_1}-e^{-t_0}}\leq
\liminf_{B\to0+0}\frac{G(t_0)-G(t_0+B)}{e^{-t_0}-e^{-t_0-B}}.$$
\end{Lemma}

\begin{proof}
Lemma \ref{l:4} tells that $G(t)<+\infty$ for any $t\geq t_1$. By Lemma \ref{l:3},
there exists a
holomorphic function $f_{t_0}$ on $\{\Psi<-t_0\}$, such that
$(f_{t_0}-f)_o\in J$ and
$\int_{\{\Psi<-t_0\}}|f_{t_0}|^{2}=G(t_0)$.

It suffices to consider that $\liminf_{B\to0+0}\frac{G(t_0)-G(t_0+B)}{e^{-t_0}-e^{-t_0-B}}\in[0,+\infty)$
because of the decreasing property of $G(t)$.
Then there exists $B_{j}\to 0+0$ $(j\to+\infty)$ such that
\begin{equation}
	\label{eq:211106d}\begin{split}
	e^{t_0}\lim_{j\rightarrow+\infty}\frac{G(t_0)-G(t_0+B_j)}{B_j}&=\lim_{j\to+\infty}\frac{G(t_0)-G(t_0+B_j)}{e^{-t_0}-e^{-t_0-B_j}}\\
	&=\liminf_{B\to0+0}\frac{G(t_0)-G(t_0+B)}{e^{-t_0}-e^{-t_0-B}},\end{split} \end{equation}
and $\left\{\frac{G(t_0)-G(t_0+B_j)}{B_j}\right\}_{j\in\mathbb{Z}_{>0}}$ is bounded.

Since $t\leq v_{t_0,B_j}(t)$, Lemma \ref{l:L2'} shows that for any $B_{j}$, there exists a holomorphic function $\tilde{F}_{j}$ on $\{\Psi<-t_1\}$, such that
\begin{equation}
\label{equ:GZc}
\begin{split}
&\int_{\{\Psi<-t_1\}}|\tilde{F}_{j}-(1-b_{t_{0},B_{j}}(\Psi))f_{t_{0}}|^{2}
\\\leq&\int_{\{\Psi<-t_1\}}|\tilde{F}_{j}-(1-b_{t_{0},B_{j}}(\Psi))f_{t_{0}}|^{2}e^{-\Psi+v_{t_0,B_j}(\Psi)}
\\\leq&
(e^{-t_1}-e^{-t_0-B_j})
\int_{\{\Psi<-t_1\}}\frac{1}{B_{j}}(\mathbb{I}_{\{-t_{0}-B_{j}<\Psi<-t_{0}\}})|f_{t_{0}}|^{2}e^{-\Psi}
\\\leq&
(e^{t_0-t_1+B_j}-1)\left(\int_{\{\Psi<-t_0\}}\frac{1}{B_{j}}|f_{t_{0}}|^{2}
-\int_{\{\Psi<-t_0-B_j\}}\frac{1}{B_{j}}|f_{t_{0}}|^{2}\right)
\\\leq&
(e^{t_0-t_1+B_j}-1)\frac{G(t_{0})-G(t_{0}+B_{j})}{B_{j}}.
\end{split}
\end{equation}
Note that $b_{t_0,B_j}(t)=0$ for $t\le-t_0-B_j$ and $b_{t_0,B_j}(t)=1$ for $t\ge t_0$. According to inequality \eqref{equ:GZc}, we have $(\tilde F_j-f_{t_0})_o\in J$, and
\begin{equation}
	\label{eq:0216a}
	\begin{split}
		&\int_{\{\Psi<-t_1\}}|\tilde F_j|^2\\
		\le&2\int_{\{\Psi<-t_1\}}|(1-b_{t_{0},B_{j}}(\Psi))f_{t_{0}}|^{2}+2\int_{\{\Psi<-t_1\}}|\tilde{F}_{j}-(1-b_{t_{0},B_{j}}(\Psi))f_{t_{0}}|^{2}\\
		\le& 2\int_{\{\Psi<-t_0\}}|f_{t_0}|^2+2\left(e^{t_0-t_1+B_j}-1\right)\frac{G(t_{0})-G(t_{0}+B_{j})}{B_{j}}.
	\end{split}
\end{equation}
Following from Lemma \ref{l:converge} and inequality \eqref{eq:0216a}, we get that there  exists a subsequence of $\{\tilde F_j\}_{j\in \mathbb{Z}_{>0}}$ (also denoted by $\{\tilde F_j\}_{j\in \mathbb{Z}_{>0}}$), which compactly converges to a  holomorphic function $\tilde F_0$ on $\{\Psi<-t_1\}$, such that $(\tilde F_0-f)_o\in J$ and
$$\int_{\{\Psi<-t_1\}}|\tilde F_0|^2\le\liminf_{j\rightarrow+\infty}\int_{\{\Psi<-t_1\}}|\tilde F_j|^2<+\infty.$$  Note that
\begin{equation*}
	\lim_{j\rightarrow+\infty}b_{t_0,B_j}(t)=\lim_{j\rightarrow+\infty}\int_{-\infty}^{t}\frac{1}{B_j}\mathbb{I}_{\{-t_0-B_j<s<-t_0\}}ds=\left\{ \begin{array}{lcl}
	0 & \mbox{if}& x\in(-\infty,-t_0)\\
	1 & \mbox{if}& x\in[-t_0,+\infty)
    \end{array}, \right.
\end{equation*}
and
\begin{equation*}
	\lim_{j\rightarrow+\infty}v_{t_0,B_j}(t)=\lim_{j\rightarrow+\infty}\int_{-t_0}^{t}b_{t_0,B_j}ds-t_0=\left\{ \begin{array}{lcl}
	-t_0 & \mbox{if}& x\in(-\infty,-t_0)\\
	t & \mbox{if}& x\in[-t_0,+\infty)
    \end{array}. \right.
\end{equation*}
According to inequality \eqref{equ:GZc} and Fatou's Lemma, we have
\begin{equation}
	\label{eq:211106b}\begin{split}
		&\int_{\{\Psi<-t_0\}}|\tilde F_0-f_{t_0}|+\int_{\{-t_0\le\Psi<-t_1\}}|\tilde F_0|^2\\
		=&\int_{\{\Psi<-t_1\}}\lim_{j\rightarrow+\infty}|\tilde{F}_{j}-(1-b_{t_{0},B_{j}}(\Psi))f_{t_{0}}|^{2}
\\
\le&\liminf_{j\rightarrow+\infty}\int_{\{\Psi<-t_1\}}|\tilde{F}_{j}-(1-b_{t_{0},B_{j}}(\Psi))f_{t_{0}}|^{2}
\\
\leq&
\left(e^{t_0-t_1}-1\right)\frac{G(t_{0})-G(t_{0}+B_{j})}{B_{j}}.
	\end{split}
\end{equation}
Then it follows from Lemma \ref{l:3}, equality \eqref{eq:211106d} and inequality \eqref{eq:211106b} that
\begin{equation*}
	\begin{split}&(e^{-t_1}-e^{-t_0})\liminf_{B\rightarrow 0+0}\frac{G(t_0)-G(t_0+B)}{e^{-t_0}-e^{-t_0-B}}\\
	=&(e^{t_0-t_1}-1)\frac{G(t_{0})-G(t_{0}+B_{j})}{B_{j}}\\
	\ge&	\int_{\{\Psi<-t_0\}}|\tilde F_0-f_{t_0}|+\int_{\{-t_0\le\Psi<-t_1\}}|\tilde F_0|^2\\
	=&\int_{\{\Psi<-t_1\}}|\tilde F_0|^2-\int_{\{\Psi<-t_0\}}|f_{t_0}|^2\\
	\ge& G(t_1)-G(t_0).
	\end{split}
\end{equation*}
This proves Lemma \ref{l:5}.
\end{proof}

The following well-known property of concave functions will be used in the proof of Theorem \ref{thm:main}.
\begin{Lemma}[see \cite{G16}]
\label{lem:Ea}
Let $a(r)$ be a lower semicontinuous function on $(0,R]$.
Then $a(r)$ is concave if and only if
\begin{equation}
	\label{eq:concave}
	\frac{a(r_{1})-a(r_{2})}{r_{1}-r_{2}}\geq \limsup_{r\to r_{2}+0}\frac{a(r)-a(r_{2})}{r-r_{2}},
\end{equation}
holds for any $0<r_{2}<r_{1}\le R$.
\end{Lemma}

\section{Proof of Theorem \ref{thm:main}}

Firstly, we prove that  $G(t_1)<+\infty$ for any $t_1\in[0,t_0)$.  It follows from Lemma \ref{l:3} that there exists a holomorphic function $f_{t_0}$ on $\{\Psi<-t_0\}$ satisfying $(f_{t_0}-f)_o\in J$ and $\int_{\{\Psi<-t_0\}}|f_{t_0}|^2=G(t_0)<+\infty$. Using Lemma \ref{l:L2'},  we get a holomorphic function $\tilde F$ on $\{\Psi<-t_1\}$, such that
\begin{equation}
	\label{eq:202142a}
	\begin{split}
		&\int_{\{\Psi<-t_1\}}|\tilde F-(1-b_{t_0,B}(\Psi))f_{t_0}|^2\\
		\leq&\int_{\{\Psi<-t_1\}}|\tilde F-(1-b_{t_0,B}(\Psi))f_{t_0}|^2e^{-\Psi+v_{t_0,B}(\Psi)}\\
		\leq&(e^{-t_1}-e^{-t_0-B})\int_{\{\Psi<-t_1\}}\frac{1}{B}\mathbb{I}_{\{-t_0-B<\Psi<-t_0\}}|f_{t_0}|^2e^{-\Psi},
	\end{split}
\end{equation}
 which implies that $(\tilde F-f_{t_0})_o\in J$ and
 \begin{displaymath}
 	\begin{split}
 		&\int_{\{\Psi<-t_1\}}|\tilde F|^2\\
 		\le &2\int_{\{\Psi<-t_1\}}|(1-b_{t_0,B}(\Psi))f_{t_0}|^2+2\int_{\{\Psi<-t_1\}}|\tilde F-(1-b_{t_0,B}(\Psi))f_{t_0}|^2\\
 		\le &2\int_{\{\Psi<-t_1\}}|f_{t_0}|^2+\frac{2(e^{t_0+B-t_1}-1)}{B}\int_{\{\Psi<-t_0\}}|f_{t_0}|^2\\
 		<&+\infty.
 	\end{split}
 \end{displaymath}
 Then we obtain $G(t_1)\leq\int_{\{\Psi<-t_1\}}|\tilde F|^2<+\infty$.

Now, as $G(h^{-1}(r))$ is lower semicontinuous (Lemma \ref{l:4}), Lemma \ref{l:5} and Lemma \ref{lem:Ea} imply the concavity of $G(-\log r)$. Lemma \ref{l:4} shows $\lim_{t\rightarrow +\infty}G(t)=0$. Hence the proof of Theorem \ref{thm:main} is completed.

\section{Proofs of Theorem \ref{thm:J_M} and Remark \ref{r:1}}

In this section, we prove Theorem \ref{thm:J_M} and Remark \ref{r:1}.

\subsection{Proof of Theorem \ref{thm:J_M}}
\

As $D$ is a pseudoconvex domain, there exist a sequence of pseudoconvex domains $D_1\Subset\ldots D_j\Subset D_{j+1}\Subset\ldots$ such that $\bigcup_{j=1}^{+\infty}D_j=D$. Denote
\begin{displaymath}
 	\inf\left\{\int_{\{p\Psi_1<-t\}\cap D_j}|\tilde f|^2:\tilde f\in\mathcal{O}(\{p\Psi_1<-t\}\cap D_j) \ \& \ (\tilde f-f)_o\in I(p\Psi_1)_o\right\}
 \end{displaymath}
by $G_{j,p}(t)$, where $t\in[0,+\infty)$, $j\in\mathbb{Z}_{>0}$ and $p\in(1,2)$.
Note that
$$p\Psi_1=\min\big\{(2pc_o^{fF}(\psi)\psi+(4-2p)\log|F|)-2\log|F^2|,0\big\},$$
 and
$$G_{j,p}(0)\le\int_{D_j}|f|^2<+\infty.$$ 
Theorem \ref{thm:main} shows that $G_{j,p}(-\log r)$ is concave with respect to $r\in(0,1]$, and $\lim_{t\rightarrow+\infty}G_{j,p}(t)=0$, yielding that
 \begin{equation}
 	\label{eq:0221b}
 	\frac{1}{r_1^2}\int_{\{p\Psi_1<2\log r_1\}\cap D_j}|f|^2\ge \frac{1}{r_1^2}G_{j,p}(2\log r_1)\ge  G_{j,p}(0),
 \end{equation}
where  $0<r_1\le 1$.

Denote
$$C_j:=\inf\left\{\int_{\{\Psi_1<0\}\cap D_j}|\tilde f|^2:\tilde f\in\mathcal{O}(\{\Psi_1<0\}\cap D_j) \ \& \ (\tilde f-f)_o\in I_+(\Psi_1)_o\right\}.$$
Since $\int_{D_j}|f|^2<+\infty$ and $G_{j,p}(0)\ge C_j$ for any $j\in\mathbb{Z}_{>0}$, according to the dominated convergence theorem and inequality \eqref{eq:0221b}, we  obtain that
\begin{equation}
	\label{eq:0221c}\begin{split}
		&\frac{1}{r_1^2}\int_{\{\Psi_1\le2\log r_1\}\cap D_j}|f|^2\\
		=&\lim_{p\rightarrow1+0}\frac{1}{r_1^2}\int_{\{p\Psi_1<2\log r_1\}\cap D_j}|f|^2\\
		\ge&\limsup_{p\rightarrow1+0}G_{j,p}(0)\\
		\ge&C_j
	\end{split}
\end{equation}
holds for any $j\in\mathbb{Z}_{>0}$ and $r_1\in(0,1]$.
Without loss of generality, assume that there exists $r_0\in(0,1]$ such that $\frac{1}{r_0^2}\int_{\{\Psi_1\le2\log r_0\}}|f|^2<+\infty$. Following from inequality \eqref{eq:0221c}, we get that $\sup_{j}C_j<+\infty$.

Lemma \ref{l:m5} tells us that there exists $p_0\in(1,2)$ such that $I_+(\Psi_1)_o=I(p_0\Psi_1)_o$. Then it follows from  Lemma \ref{l:3} that there exists a holomorphic function $\tilde f_j$ on $\{\Psi_1<0\}\cap D_j$ such that $C_j=\int_{\{\Psi_1<0\}\cap D_j}|\tilde f_j|^2$ and $(\tilde f_j-f)_o\in I(p_0\Psi_1)_o$ for any $j\in\mathbb{Z}_{>0}$. Note that $\sup_{j}\int_{\{\Psi_1<0\}\cap D_j}|\tilde f_j|^2<+\infty$. Then we can apply Lemma \ref{l:converge} and the diagonal method to extract a subsequence of $\{\tilde f_j\}_{j\in\mathbb{Z}_{>0}}$ (also denoted by $\{\tilde f_j\}_{j\in\mathbb{Z}_{>0}}$) compactly convergent to a holomorphic function $\tilde f_0$ on $\{\Psi_1<0\}$, which satisfies $(\tilde f_0-f)_o\in I(p_0\Psi)_o$. As $C_j$ is increasing with respect to $j$,  using Fatou's Lemma and the definition of $G(0;\Psi_1,I_+(\Psi_1)_o,f)$, we obtain that
\begin{equation}\label{eq:0221g}
	G(0;\Psi_1,I_+(\Psi_1)_o,f)\le\int_{\{\Psi<0\}}|\tilde f_0|^2\le\liminf_{j\rightarrow+\infty}\int_{\{\Psi_1<0\}\cap D_j}|\tilde f_j|^2 = \lim_{j\rightarrow+\infty}C_j.
\end{equation}
Combining inequality \eqref{eq:0221c} and inequality \eqref{eq:0221g}, we get
\begin{equation}
	\label{eq:0221h}
	\begin{split}
		&\frac{1}{r_1^2}\int_{\{\Psi_1\le2\log r_1\}}|f|^2\\
		\ge&\lim_{j\rightarrow+\infty}\frac{1}{r_1^2}\int_{\{\Psi_1\le2\log r_1\}\cap D_j}|f|^2\\
		\ge&\lim_{j\rightarrow+\infty}C_j\\
		\ge&G(0;\Psi_1,I_+(\Psi_1)_o,f).
	\end{split}
\end{equation}
Note that
\[\big\{2c_o^{fF}(\psi)\psi-2\log|F|<2\log r\big\}=\bigcup_{0<r_1<r}\{\Psi_1\le2\log r_1\}\]
where $r\in(0,1]$. Then inequality \eqref{eq:0221h} implies
\begin{flalign*}
	\begin{split}
		\int_{\{c_o^{fF}(\psi)\psi-\log|F|<\log r\}}|f|^2&\ge \sup_{r_1\in(0,r)}r_1^2G(0;\Psi_1,I_+(\Psi_1)_o,f)\\
		&=r^2G(0;\Psi_1,I_+(\Psi_1)_o,f).
	\end{split}
\end{flalign*}
Lemma \ref{l:m6} shows $f_o\not\in I_+(\Psi_1)_o$.
In addition, $I_+(\Psi_1)_o=I(p_0\Psi_1)_o$ and Lemma \ref{l:2} verify $G(0;\Psi_1,I_+(\Psi_1)_o,f)>0$.
Thus, Theorem \ref{thm:J_M} holds.

\subsection{Proof of Remark \ref{r:1}}\label{sec:p-r}
\

Note that $G(0;\Psi_1,I_+(\Psi_1)_o,1)\ge G(0;\Psi_1,I(\Psi_1)_o,1)$. Then it suffices to prove
\[G(0;\Psi_1,I(\Psi_1)_o,1)\ge \frac{C_{F^{1+\delta},2c_o^F(\psi)\psi+\delta\max\{2c_o^{F}(\psi)\psi,2\log|F|\}}(o)}{\left(1+\frac{1}{\delta}\right)\sup_{D}e^{(1+\delta)\max\{2c_o^{F}(\psi)\psi,2\log|F|\}}},\]
for any $\delta\in\mathbb{Z}_{>0}$. Without loss of generality, assume that $G(0;\Psi_1,I(\Psi_1)_o,1)<+\infty$. Denote $G(t;\Psi_1,I(\Psi_1)_o,1)$ by $G(t)$ for short.

Lemma \ref{l:3} indicates that there exists a holomorphic function $f_0$ on $\{\Psi_1<0\}$ such that $(f_0-1)_o\in I(\Psi_1)_o$ and $G(0)=\int_{\{\Psi_1<0\}}|f_0|^2$. Denote
$$\varphi:=(1+\delta)\max\{2c_o^F(\psi)\psi,2\log|F|\}.$$
For any $B>0$, it follows from Lemma \ref{l:L2} that there exists a holomorphic function $\tilde F_B$ on $D$ such that
\begin{equation}
	\label{eq:0227a}
	\begin{split}
		&\int_{D}|\tilde F_B-(1-b_{B}(\Psi_1))f_0F^{1+\delta}|^2e^{-\varphi+v_{B}(\Psi_1)-\Psi_1}\\
		\le&\left(\frac{1}{\delta}+1-e^{-B} \right)\int_D\frac{1}{B}\mathbb{I}_{\{-B<\Psi_1<0\}}|f_0|^2e^{-\Psi_1}\\
		\le&\left(\left(\frac{1}{\delta}+1\right)e^{B}-1 \right)\int_D\frac{1}{B}\mathbb{I}_{\{-B<\Psi_1<0\}}|f_0|^2\\
		\le&\left(\left(\frac{1}{\delta}+1\right)e^{B}-1 \right)\frac{G(0)-G(B)}{B},
	\end{split}	\end{equation}
	where $b_{B}(t)=\int_{-\infty}^t\frac{1}{B}\mathbb{I}_{\{-B<s<0\}}ds$ and $v_{B}(t)=\int_0^tb_{B}(s)ds$.
Theorem \ref{thm:main} shows that $G(-\log r)$ is concave with respect to $r$ and $\lim_{t\rightarrow+\infty}G(t)=0$, which implies that
\begin{equation}
	\label{eq:0227b}
\begin{split}
		&\lim_{B\rightarrow0+0}\left(\left(\frac{1}{\delta}+1\right)e^{B}-1 \right)\frac{G(0)-G(B)}{B}\\
		=&	\lim_{B\rightarrow0+0}\left(\left(\frac{1}{\delta}+1\right)e^{B}-1 \right)\frac{G(-\log1)-G(-\log e^{-B})}{1-e^{-B}}\cdot\frac{1-e^{-B}}{B}\\
		\le&\frac{1}{\delta}G(0).
\end{split}
\end{equation}
Combining inequality \eqref{eq:0227a}, inequality \eqref{eq:0227b} and the fact $v_B(\Psi_1)-\Psi_1\ge 0$, we have
\begin{equation}
	\label{eq:0227c}\begin{split}
		&\lim_{B\rightarrow0+0}\int_D|\tilde F_B|^2e^{-\varphi}\\
		\le&\lim_{B\rightarrow0+0}2\left(\int_{D}|\tilde F_B-(1-b_{B}(\Psi_1))f_0F^{1+\delta}|^2e^{-\varphi}+\int_{D}|(1-b_{B}(\Psi_1))f_0F^{1+\delta}|^2e^{-\varphi}\right)\\
		\le&2\int_{\{\Psi_1<0\}}|f_0|^2+ 2\lim_{B\rightarrow0+0}\left(\left(\frac{1}{\delta}+1\right)e^{B}-1 \right)\frac{G(0)-G(B)}{B}\\
		\le &2\left(1+\frac{1}{\delta}\right)G(0)\\
		<&+\infty.
	\end{split}
\end{equation}
This implies that there exists a subsequence of $\{\tilde F_B\}_{B>0}$, denoted by $\{\tilde F_{B_j}\}_{j\in\mathbb{Z}_{>0}}$, satisfying that $\lim_{j\rightarrow+\infty}B_j=0$ and $\{\tilde F_{B_j}\}_{j\in\mathbb{Z}_{>0}}$ converges to a holomorphic function $\tilde F$ on $D$. It follows from Fatou's Lemma and inequality \eqref{eq:0227c} that
\begin{equation}
	\label{eq:0227d}\int_{D}|\tilde F|^2e^{-\varphi}\le\liminf_{j\rightarrow+\infty}\int_D|\tilde F_{B_j}|^2e^{-\varphi}\leq\left(1+\frac{1}{\delta}\right)G(0)<+\infty.
\end{equation}
Using Fatou's Lemma, inequality \eqref{eq:0227a} and inequality \eqref{eq:0227b}, we obtain
\begin{equation}\label{eq:0227e}
\begin{split}
		&\int_{\{\Psi_1<0\}}|\tilde F-f_0F^{1+\delta}|^2e^{-\varphi-\Psi_1}+\int_{\{\Psi_1\ge0\}}|\tilde F|^2e^{-\varphi}\\
		\le&\liminf_{j\rightarrow+\infty}\int_{D}|\tilde F_{B_j}-(1-b_{B_j}(\Psi_1))f_0F^{1+\delta}|^2e^{-\varphi+v_{B_j}(\Psi_1)-\Psi_1}\\
		\le&\liminf_{j\rightarrow+\infty}\left(\left(\frac{1}{\delta}+1\right)e^{B_j}-1 \right)\frac{G(0)-G(B_j)}{B_j}\\
		\le&\frac{1}{\delta}G(0).
\end{split}
\end{equation}

Note that
\[\int_{\{\Psi_1<0\}}\left|\frac{\tilde F}{F^{1+\delta}}\right|^2=\int_{\{\Psi_1<0\}}|\tilde F|^2e^{-\varphi}<+\infty,\]
and
\[\int_{\{\Psi_1<0\}}\left|\frac{\tilde F}{F^{1+\delta}}-f_0\right|^2e^{-\Psi_1}=\int_{\{\Psi_1<0\}}|\tilde F-f_0F^{1+\delta}|^2e^{-\varphi-\Psi_1}<+\infty.\]
Then Lemma \ref{l:3} indicates that
\begin{flalign}\label{eq:0227f}
	\begin{split}
	\int_{\{\Psi_1<0\}}|\tilde F-f_0F^{1+\delta}|^2e^{-\varphi}&=\int_{\{\Psi_1<0\}}\left|\frac{\tilde F}{F^{1+\delta}}-f_0\right|^2\\
	&=\int_{\{\Psi_1<0\}}\left|\frac{\tilde F}{F^{1+\delta}}\right|^2-\int_{\{\Psi_1<0\}}|f_0|^2.
	\end{split}
\end{flalign}
Combining inequality \eqref{eq:0227e} and equality \eqref{eq:0227f}, we get
\begin{equation}
	\label{eq:0227g}
	\begin{split}
		&\int_{D}|\tilde F|^2e^{-\varphi}\\
		=&\int_{\{\Psi_1<0\}}|\tilde F|^2e^{-\varphi}-\int_{\{\Psi_1<0\}}|f_0|^2+\int_{\{\Psi_1\ge0\}}|\tilde F|^2e^{-\varphi}+\int_{\{\Psi_1<0\}}|f_0|^2\\
		=&\int_{\{\Psi_1<0\}}|\tilde F-f_0F^{1+\delta}|^2e^{-\varphi}+\int_{\{\Psi_1\ge0\}}|\tilde F|^2e^{-\varphi}+G(0)\\
		\le&\left(1+\frac{1}{\delta}\right)G(0).
	\end{split}
\end{equation}
As $(f-1)_o\in I(\Psi_1)_o$ and
\[\int_{\{\Psi_1<0\}}\left|\frac{\tilde F}{F^{1+\delta}}-f_0\right|^2e^{-\Psi_1}=\int_{\{\Psi_1<0\}}|\tilde F-f_0F^{1+\delta}|^2e^{-\varphi-\Psi_1}<+\infty,\]
we have $(\frac{\tilde F}{F^{1+\delta}}-1)_o\in I(\Psi_1)_o$, yielding that there exist a neighborhood $U\Subset D$ of $o$ and $t>0$ such that
\begin{equation}
	\label{eq:0227h}
	\int_{\{\Psi_1<-t\}\cap U}|\tilde F-F^{1+\delta}|^2e^{-\varphi-\Psi_1}=\int_{\{\Psi_1<-t\}\cap U}\left|\frac{\tilde F}{F^{1+\delta}}-1\right|^2e^{-\Psi_1}<+\infty.
\end{equation}
Sicne
\[\int_{\{\Psi_1>-t\}\cap U}|\tilde F|^2e^{-\varphi-\Psi_1}\le e^t\int_{D}|\tilde F|^2e^{-\varphi}<+\infty,\]
and 
\[\int_{\{\Psi_1>-t\}\cap U}|F^{1+\delta}|^2e^{-\varphi-\Psi_1}\le \int_{U\cap \{\Psi_1\ge-t\}}e^{-\Psi_1}<+\infty,\]
inequality \eqref{eq:0227h} verifies that
\begin{flalign*}
	\begin{split}
		(\tilde F-F^{1+\delta},o)&\in\mathcal{I}(\varphi+\Psi_1)_o\\
		&=\mathcal{I}\Big(2c_o^F(\psi)\psi+\delta\max\big\{2c_o^{F}(\psi)\psi,2\log|F|\big\}\Big)_o.
	\end{split}
\end{flalign*}
Recall the definition of $C_{F^{1+\delta},2c_o^F(\psi)\psi+\delta\max\{2c_o^{F}(\psi)\psi,2\log|F|\}}(o)$. According to inequality \eqref{eq:0227g}, we obtain that
\begin{displaymath}
\frac{C_{F^{1+\delta},2c_o^F(\psi)\psi+\delta\max\{2c_o^{F}(\psi)\psi,2\log|F|\}}(o)}{\sup_{D}e^{(1+\delta)\max\{2c_o^F(\psi)\psi,2\log|F|\}}}\le\int_{D}|\tilde F|^2e^{-\varphi}\le\left(1+\frac{1}{\delta}\right)G(0).
\end{displaymath}
Thus, Remark \ref{r:1} holds.

\section{Proofs of  Theorem \ref{thm:J-M2} and Corollary \ref{c:3}}

In this section, we prove Theorem \ref{thm:J-M2} and Corollary \ref{c:3}.

\subsection{Proof of Theorem \ref{thm:J-M2}}
\

Lemma \ref{l:m5} shows that there exists some $p_0>2a_o^f(\Psi)$ such that $I(p_0\Psi)_o=I_+(2a_o^f(\Psi)\Psi)_o$. According to the definition of $a_o^{f}(\Psi)$ and Lemma \ref{l:2}, we obtain that
\begin{equation}
	\label{eq:0222d}G(0;\Psi,I_+(2a_o^f(\Psi)\Psi)_o,f)>0.
\end{equation}
Without loss of generality, assume that there exists $t>t_0$ such that $\int_{\{\Psi<-t\}}|f|^2<+\infty$.
Denote
\[t_1:=\inf\Big\{t\ge t_0:\int_{\{\Psi<-t\}}|f|^2<+\infty\Big\},\]
and
\begin{displaymath}
	G_{p}(t):=\inf\Big\{\int_{\{p\Psi<-t\}}|\tilde f|^2:\tilde f\in\mathcal{O}(\{p\Psi<-t\}) \ \& \ (\tilde f-f)_o\in I(p\Psi)_o\Big\},
 \end{displaymath}
where $t\in[0,+\infty)$ and $p>2a_o^f(\Psi)$.  By the definition of $G(0;\Psi,I_+(2a_o^f(\Psi)\Psi)_o,f)$, we have that $G_{p}(0)\ge G(0;\Psi,I_+(2a_o^f(\Psi)\Psi)_o,f)$ for any $p>2a_o^f(\Psi)$. Note that
$$p\Psi=\min\Big\{2\psi+(2\lceil p\rceil-2p)\log|F|-2\log\big|F^{\lceil p\rceil}\big|,0\Big\},$$
where $\lceil p\rceil=\min\{n\in \mathbb{Z}:n\geq p\}$,
 and
 $$G_{p}(pt)\le\int_{\{\Psi<-t\}}|f|^2<+\infty,$$
 for any $t>t_1$. Theorem \ref{thm:main} indicates that $G_{j,p}(-\log r)$ is concave with respect to $r\in(0,1]$ and $\lim_{t\rightarrow+\infty}G_{j,p}(t)=0$, which verifies that
 \begin{flalign}
 	\label{eq:0222b}
	\begin{split}
 	\frac{1}{r_1^2}\int_{\{p\Psi<2\log r_1\}}|f|^2\ge \frac{1}{r_1^2}G_{p}(-2\log r_1)&\ge G_{p}(0)\\
	&\ge G(0;\Psi,I_+(2a_o^f(\Psi)\Psi)_o,f),
	\end{split}
 \end{flalign}
where  $0<r_1\le e^{-\frac{pt_0}{2}}$.

We prove $a_o^f(\Psi)>0$ by contradiction: if $a_o^f(\Psi)=0$, as $\int_{\{\Psi<-t_1-1\}}|f|^2<+\infty$, the dominated convergence theorem and inequality \eqref{eq:0222b} show that
\begin{flalign}
	\label{eq:0222c}
	\begin{split}
	\frac{1}{r_1^2}\int_{\{\Psi=-\infty\}}|f|^2&=\lim_{p\rightarrow0+0}\frac{1}{r_1^2}\int_{\{p\Psi<2\log r_1\}}|f|^2\\
	&\ge G(0;\Psi,I_+(2a_o^f(\Psi)\Psi)_o,f).
	\end{split}
\end{flalign}
Note that $\mu(\{\Psi=-\infty\})=\mu(\{\psi=-\infty\})=0$, where $\mu$ is the Lebesgue measure on $\mathbb{C}^n$. Then inequality \eqref{eq:0222c} implies that $G(0;\Psi,I_+(2a_o^f(\Psi)\Psi)_o,f)=0$, which contradicts inequality \eqref{eq:0222d}. Thus, we get $a_o^f(\Psi)>0$.

For any $r_2\in (0,e^{-a_o^{f}(\Psi)t_1})$ and $p\in(2a_0^f(\Psi),-\frac{2\log r_2}{t_1})$, since $\frac{2\log r_2}{p}<-t_1$, we have $\int_{\{p\Psi<2\log r_2\}}|f|^2<+\infty$.
Then it follows from the dominated convergence theorem and inequality \eqref{eq:0222b} that
\begin{equation}
	\label{eq:0222e}\begin{split}
	\frac{1}{r_2^2}\int_{\{2a_0^f(\Psi)\Psi\le2\log r_2\}}|f|^2&=\lim_{p\rightarrow2a_0^f(\Psi)+0}\frac{1}{r_2^2}\int_{\{p\Psi<2\log r_2\}}|f|^2\\
	&\ge  G(0;\Psi,I_+(2a_o^f(\Psi)\Psi)_o,f).	
	\end{split}
\end{equation}
Set $r\in(0,e^{-a_o^f(\Psi)t_0}]$. If $r>e^{-a_o^{f}(\Psi)t_1}$, we have
\[\int_{\{a_0^f(\Psi)\Psi<\log r\}}|f|^2=+\infty>G(0;\Psi,I_+(2a_o^f(\Psi)\Psi)_o,f).\]
If $r\in(0,e^{-a_o^f(\Psi)t_1}]$, it follows from \[\{a_o^f(\Psi)\Psi<\log r\}=\bigcup_{0<r_2<r}\{a_o^f(\Psi)\Psi<\log r_2\}\]
and inequality \eqref{eq:0222e} that
\begin{displaymath}
	\begin{split}
\int_{\{a_0^f(\Psi)\Psi<\log r\}}|f|^2&=
\sup_{r_2\in(0,r)}\int_{\{2a_0^f(\Psi)\Psi\le2\log r_2\}}|f|^2
\\&\ge \sup_{r_2\in(0,r)}r_2^2G(0;\Psi,I_+(2a_o^f(\Psi)\Psi)_o,f)\\
&=r^2G(0;\Psi,I_+(2a_o^f(\Psi)\Psi)_o,f).
	\end{split}
\end{displaymath}

Thus, Theorem \ref{thm:J-M2} holds.

\subsection{Proof of  Corollary \ref{c:3}}
\

It is clear that $I_+(a\Psi)_o\subset I(a\Psi)_o$. Hence, it suffices to prove $I(a\Psi)\subset I_+(a\Psi)_o$.

If there exists a holomorphic function $f$ on $\{\Psi<-t_0\}\cap D_0$ such that $f_o\in I(a\Psi)_o$ and $f_o\not\in I_+(a\Psi)_o$, where $t_0>0$ and $D_0\subset D$ is a neighborhood of $o$, then $a_o^f(\Psi)_o=a/2<+\infty$.  Theorem \ref{thm:J-M2} shows that $a>0$. For any neighborhood $U\subset D_0$ of $o$, it follows from Theorem \ref{thm:J-M2} that  there exists $C_U>0$ such that
\begin{equation}
	\label{eq:0222f}\frac{1}{r^2}\int_{\{a\Psi<2\log r\}\cap U}|f|^2\ge C_U,
\end{equation}
for any $r\in(0,e^{-\frac{at_0}{2}}]$. Given any $t>at_0$, Fubini's Theorem and inequality \eqref{eq:0222f} imply that
\begin{displaymath}
	\begin{split}
		\int_{\{a\Psi<-t\}\cap U}|f|^2e^{-a\Psi}
		=&\int_{\{a\Psi<-t\}\cap U}\left(|f|^2\int_0^{e^{-a\Psi}}dl\right)\\
		=&\int_0^{+\infty}\left(\int_{\{l<e^{-a\Psi}\}\cap\{a\Psi<-t\}\cap U}|f|^2\right)dl\\
		\ge&\int_{e^t}^{+\infty}\left(\int_{\{a\Psi<-\log l\}\cap U}|f|^2\right)dl\\
		\ge&C_U\int_{e^t}^{+\infty}\frac{1}{l}dl\\
		=&+\infty,
	\end{split}
\end{displaymath}
which contradicts  $f_o\in I(a\Psi)_o$. Thus, there does not exist $f_o\in I(a\Psi)_o$ such that $f_o\not\in I_+(a\Psi)_o$, i.e. $I(a\Psi)= I_+(a\Psi)_o$ for any $a\ge0$.

\subsubsection{Another proof of Corollary \ref{c:3}}
\

We give another proof of Corollary \ref{c:3} by using Lemma \ref{l:m4} and the strong openness property of multiplier ideal sheaves.

According to Lemma \ref{l:m4}, we only need to prove that
$$\mathcal{I}(k\varphi+a\Psi)_o=\bigcup_{p\in(a,k)}\mathcal{I}(k\varphi+p\Psi)_o,$$
where $\varphi=2\max\{\psi,2\log|F|\}$, and $k>a$ is an integer. It is easy to see
\[\bigcup_{p\in(a,k)}\mathcal{I}(k\varphi+p\Psi)_o\subset \mathcal{I}(k\varphi+a\Psi)_o.\] In the following, we prove that \[\mathcal{I}(k\varphi+a\Psi)_o\subset \bigcup_{p\in(a,k)}\mathcal{I}(k\varphi+p\Psi)_o.\]
Note that $k\varphi+p\Psi$ is a plurisubharmonic function on $D$ for any $p\in[0,k]$.

When $a>0$, for any $(f,o)\in \mathcal{I}(k\varphi+a\Psi)_o$, it follows from the strong openness property of multiplier ideal sheaves  (\cite{GZSOC}) that there exists $r>1$ such that $(f,o)\in \mathcal{I}(rk\varphi+ar\Psi)_o$.
As
\[\mathcal{I}(rk\varphi+ar\Psi)_o\subset \bigcup_{p\in(a,k)}\mathcal{I}(k\varphi+p\Psi)_o,\]
we have
\[(f,o)\in \bigcup_{p\in(a,k)}\mathcal{I}(k\varphi+p\Psi)_o.\]

When $a=0$, for any $(f,o)\in \mathcal{I}(k\varphi+a\Psi)_o$, as $\psi$ is plurisubharmonic,  it follows from the strong openness property of multiplier ideal sheaves  (\cite{GZSOC}) and H\"older's inequality that there exist $r_1>1$ and a neighborhood $U\Subset D$ of $o$ satisfying that
\begin{equation}
	\label{eq:0222g}\begin{split}
	\int_{U}|f|^2e^{-k\varphi-r_1\psi}\le \left(\int_{U}|f|^{2q_1}e^{-q_1k\varphi}\right)^{\frac{1}{q_1}}\cdot\left(\int_Ue^{-q_2r_1\psi} \right)^{\frac{1}{q_2}}	<+\infty,
	\end{split}\end{equation}
	where $q_1>1$ and $q_2>1$ satisfy $\frac{1}{q_1}+\frac{1}{q_2}=1.$
Since $\varphi/2+\Psi=\psi$, inequality \eqref{eq:0222g} implies
\[(f,o)\in\bigcup_{p\in(a,k)}\mathcal{I}(k\varphi+p\Psi)_o.\]

Thus, Corollary \ref{c:3} holds.

	  \section{Proof of Proposition \ref{p:semi}}
	Denote 
	\[C:=\inf_{m}G(0;\Psi_m,I(\Psi_m)_o,f_m)>0.\]
	It follows from Theorem \ref{thm:main} that $G(-\log r;\Psi_m,I(\Psi_m)_o,f_m)$ is concave with respect to $r\in(0,1]$. Then for some $t_0>0$, it holds that
	\begin{equation}
		\label{eq:0228a}
		\int_{\{\Psi_m<-t\}}|f_m|^2\ge e^{-t}G(0;\Psi_m,I(\Psi_m)_o,f_m)\ge e^{-t}C, \ \forall t\ge t_0.
	\end{equation}
For any $\epsilon>0$, as $\{\Psi_m\}_{m\in\mathbb{Z}_{>0}}$	converges to $\Psi$, there exists $m_0>0$ such that
\begin{equation}
	\label{eq:0228b}
	\mu\{|\Psi_m-\Psi|\ge1\}<\epsilon, \ \forall m\ge m_0,
\end{equation}
where $\mu$ is the Lebesgue measure on $\mathbb{C}^n$. Denote
\[M:=\sup_{m}\sup_D|\tilde f_m|^2<+\infty.\]
Note that $\tilde f_m=f_m$ on $\{\Psi_m<-t_0\}.$ Then inequality \eqref{eq:0228a} and inequality \eqref{eq:0228b} indicate that
	  \begin{equation*}
	  	\begin{split}
	  		&\int_{\{\Psi<-t+1\}}|\tilde f_m|^2\\
	  		\ge&\int_{\{\Psi_m<-t\}}|\tilde f_m|^2-\int_{\{|\Psi_m-\Psi|\ge1\}}|\tilde f_m|^2\\
	  		\ge&e^{-t}C-M\epsilon,
	  	\end{split}
	  \end{equation*}
for any $t\ge t_0$ and $m\ge m_0$, yielding that
	  \begin{equation}
	  	\label{eq:0228c}
		\liminf_{m\rightarrow+\infty}\int_{\{\Psi<-t+1\}}|\tilde f_m|^2\ge e^{-t}C, \ \forall t\ge t_0.
	  \end{equation}
	  As $\{\tilde f_m\}_{m\in\mathbb{Z}_{>0}}$ converges to $\tilde f$ in Lebesgue measure, and $\sup_{m}\sup_D|\tilde f_m|<+\infty$, the dominated convergence theorem and inequality \eqref{eq:0228c} verify that 
	  \begin{equation*}
	  	\int_{\{\Psi<-t+1\}}|\tilde f|^2\ge e^{-t}C, \ \forall t\ge t_0.
	  \end{equation*}
Using Fubini's Theorem, we get that 
	  \begin{displaymath}
	  	\begin{split}
	  		\int_D|\tilde{f}|^2e^{-\Psi}\ge\int_{\{\Psi<-t_0\}}|\tilde f|^2e^{-\Psi}&=\int_{\{\Psi<-t_0\}}\left(|\tilde f|^2\int_0^{e^{-\Psi}}dl \right)\\
	  		&=\int_0^{+\infty}\left(\int_{\{l<e^{-\Psi}\}\cap\{\Psi<-t_0\}}|\tilde f|^2 \right)dl\\
	  		&\ge\int_{e^{t_0}}^{+\infty}\left(\int_{\{\Psi<-\log l\}}|\tilde f|^2 \right)dl\\
	  		&\ge e^{-1}C\int_{e^{t_0}}^{+\infty}\frac{1}{l}dl\\
	  		&=+\infty,
	  	\end{split}
	  \end{displaymath}
which holds for any pseudoconvex domain $D\subset\Delta^n$ containing $o$. Consequently, we conclude that
$$|\tilde f|^2e^{-\Psi}\not\in L^1(U),$$
where $U$ is any neighborhood of $o$.

\section{Appendix}

\subsection{Proof of Lemma 2.1}
\

In the first part of the appendix, we give the proof of Lemma \ref{l:L2}.

Firstly, we do some preparations. We recall some lemmas on $L^2$ estimates for $\bar\partial$-equations. In the following, $\bar\partial^*$ denotes the Hilbert adjoint operator of $\bar\partial$.

\begin{Lemma}
	\label{l23}
	(see \cite{Siu96}, see also \cite{berndtsson})
	Let $\Omega\Subset\mathbb C^n$ be a domain with $C^{\infty}$ boundary $b\Omega$, $\Phi\in C^{\infty}(\bar\Omega)$. Let $\rho$ be a  $C^{\infty}$ defining  function for $\Omega$ such that $|d\rho|=1$ on $b\Omega$. Let $\eta$ be a smooth function on $\bar\Omega$. For any $(0,1)$-form $\alpha=\sum_{j=1}^{n}\alpha_{\bar j}d\bar z^{j}\in Dom_{\Omega}(\bar\partial^*)\cap C_{(0,1)}^{\infty}(\bar\Omega)$,
	\begin{equation}
		\label{eq:l231}
		\begin{split}
			&\int_{\Omega}\eta|\bar\partial_{\Phi}^*\alpha|^2e^{-\Phi}+\int_{\Omega}\eta|\bar\partial\alpha|^2e^{-\Phi}=\sum_{i,j=1}^n\int_{\Omega}\eta|\bar\partial_{i}\alpha_{\bar j}|^2e^{-\Phi}\\
			&+\sum_{i,j=1}^n\int_{b\Omega}\eta(\partial_i\bar\partial_j\rho)\alpha_{\bar i}\bar\alpha_{\bar j}e^{-\Phi}+\sum_{i,j=1}^n\int_{\Omega}\eta(\partial_i\bar\partial_j\Phi)\alpha_{\bar i}\bar\alpha_{\bar j}e^{-\Phi}\\
			&+\sum_{i,j=1}^n\int_{\Omega}-(\partial_i\bar\partial_j\eta)\alpha_{\bar i}\bar\alpha_{\bar j}e^{-\Phi}+2\mathrm{Re}(\bar\partial_{\Phi}^*\alpha,\alpha\llcorner(\bar\partial\eta)^{\sharp})_{\Omega,\Phi},
							\end{split}
	\end{equation}
	where $\alpha\llcorner(\bar\partial\eta)^{\sharp}=\sum_j\alpha_{\bar j}\partial_j\eta$.
	
		\end{Lemma}

The symbols and notations can be referred to \cite{guan-zhou13ap}. See also \cite{Siu96} or \cite{siu00}.

\begin{Lemma}
	\label{l24}
	(see \cite{berndtsson}, see also \cite{guan-zhou13ap}) Let $\Omega\Subset\mathbb C^n$ be a strictly pseudoconvex domain with $C^{\infty}$ boundary $b\Omega$ and $\Phi\in C(\bar\Omega)$. Let $\lambda$ be a $\bar\partial$ closed smooth form of bidegree $(n,1)$ on $\bar\Omega$. Assume the inequality
	$$|(\lambda,\alpha)_{\Omega,\Phi}|^2\leq C\int_{\Omega}|\bar\partial_{\Phi}^*\alpha|^2\frac{e^{-\Phi}}{\mu}<+\infty,$$
	where $\frac{1}{\mu}$ is an integrable positive function on $\Omega$ and $C$ is a constant, holds for all $(n,1)$-form $\alpha\in Dom_{\Omega}(\bar\partial^*)\cap Ker(\bar\partial)\cap C_{(n,1)}^{\infty}(\bar\Omega)$. Then there is a solution $u$ to the equation $\bar\partial u=\lambda$ such that
	$$\int_{\Omega}|u|^2\mu e^{-\Phi}\leq C.$$
\end{Lemma}

The following lemma will be used in the proof of Lemma \ref{l:appro}.
\begin{Lemma}
	[see \cite{Demaillybook}]\label{l:demailly}
	For arbitrary $\eta=(\eta_1,\ldots,\eta_{p})\in(0,+\infty)^p,$ the function
	$$M_{\eta}(t_1,\ldots,t_p)=\int_{\mathbb{R}^p}\max\{t_1+h_1,\ldots,t_p+h_p\}\prod_{1\le j\le p}\theta\left(\frac{h_j}{\eta_j}\right)dh_1\ldots dh_{p}$$
	possesses the following properties:
	
	$(1)$ $M_{\eta}(t_1,\ldots,t_p)$ is non decreasing in all variables, smooth and convex on $\mathbb{R}^p$;
	
	$(2)$ $\max\{t_1,\ldots,t_p\}\leq M_{\eta}(t_1,\ldots,t_p)\le\max\{t_1+\eta_1,\ldots,t_p+\eta_p\}$;
	
	$(3)$ $M_{\eta}(t_1,\ldots,t_p)=M_{(\eta_1,\ldots,\hat{\eta}_j,\ldots,\eta_p)}(t_1,\ldots,\hat{t}_j,\ldots,t_p)$ if $t_j+\eta_j\le\max_{k\not=j}\{t_k-\eta_k\}$;
	
	$(4)$ $M_{\eta}(t_1+a,\ldots,t_p+a)=M_{\eta}(t_1,\ldots,t_p)+a$ for any $a\in\mathbb{R}$;
	
	$(5)$ if $u_1,\ldots,u_p$ are plurisubharmonic functions, then $u=M_{\eta}(u_1,\ldots,u_p)$ is plurisubharmonic.
\end{Lemma}

Following the notations and assumptions in Lemma \ref{l:L2}, we present the following approximation property of $\varphi$ and $\Psi$.
\begin{Lemma}
	\label{l:appro}For any $t_0>0$, there exist smooth plurisubharmonic functions $\{\varphi_m\}_{m\in\mathbb{Z}_{>0}}$ and  smooth functions $\{\Psi_m\}_{m\in\mathbb{Z}_{>0}}$ on $D\backslash\{F=0\}$, such that
	
	$(1)$ $\{\varphi_m+\Psi_m\}_{m\in\mathbb{Z}_{>0}}$ and $\{\varphi_m+(1+\delta)\Psi_m\}_{m\in\mathbb{Z}_{>0}}$ are smooth plurisubharmonic functions;
	
	$(2)$ the sequence $\{\varphi_m\}_{m\in\mathbb{Z}_{>0}}$  is  convergent to $\varphi$, and there exists a smooth plurisubharmonic function $\varphi_0$ on $D\backslash\{F=0\}$ such that $\varphi_0\ge\varphi_m\ge\varphi$ on $D\backslash\{F=0\}$ for any $m\in\mathbb{Z}_{>0}$;
	
	$(3)$ the sequence $\{\varphi_m+(1+\delta)\Psi_m\}_{m\in\mathbb{Z}_{>0}}$ is decreasingly convergent to $\varphi+(1+\delta)\Psi$;
	
	$(4)$ $\{\Psi_m<-t_0\}\subset\{\Psi<-t_0\}$ and $\Psi_m\le0$ on $D\backslash\{F=0\}$ for any $m\in\mathbb{Z}_{>0}$.
\end{Lemma}
\begin{proof}
	As $D$ is a pseudoconvex domain, there exists a sequence of plurisubharmonic functions $\{\psi_m\}_{m\in\mathbb{Z}_{>0}}$, which is decreasingly convergent to $\psi$.
	Let $\eta_m=(\frac{t_0}{3m},\frac{t_0}{3m})$. Set
	$$\varphi_m=(1+\delta)M_{\eta_m}(\psi_m,2\log|F|),$$
	 and
	 $$\Psi_m=\psi_m-\frac{\varphi_m}{1+\delta},$$
	which are smooth functions on $D\backslash\{F=0\}$.

	Now, we prove that $(\varphi_m,\Psi_m)$ satisfies the four conditions. Lemma \ref{l:demailly} shows that $\varphi_m$ is plurisubharmonic. Note that
	\[\varphi_m+\Psi_m=\psi_m+\frac{\delta}{1+\delta}\varphi_m\]
	 and
	 \[\varphi_m+(1+\delta)\Psi_m=(1+\delta)\psi_m\]
	  are smooth plurisubharmonic functions, and $\{\varphi_m+(1+\delta)\Psi_m\}_{m\in\mathbb{Z}_{>0}}$ decreasingly converges to $\varphi+(1+\delta)\Psi=(1+\delta)\psi$. Set
	  \[\varphi_0=(1+\delta)\left(\max\{\psi_1,2\log|F|\}+\frac{t_0}{3}\right).\]
	It follows from Lemma \ref{l:demailly} that
	  $$\varphi\le(1+\delta)\max\{\psi_m,2\log|F|\}\le\varphi_m\le(1+\delta)\left(\max\{\psi_m,2\log|F|\}+\frac{t_0}{3m}\right),$$
which implies that $\{\varphi_m\}_{m\in\mathbb{Z}_{>0}}$ is convergent to $\varphi$ and $\varphi\le\varphi_m\le\varphi_0$ on $D\backslash\{F=0\}$. It can be verified by Lemma \ref{l:demailly} that:

(1) if $\psi_m\le2\log|F|-\frac{2t_0}{3m}$ holds, then $\varphi_m=2(1+\delta)\log|F|$ and $\Psi_m=\psi_m-2\log|F|\le0$;

(2) if $\psi_m\ge 2\log|F|+\frac{2t_0}{3m}$ holds, then $\varphi_m=(1+\delta)\psi_m$ and $\Psi_m=0$;

(3) if $2\log|F|-\frac{2t_0}{3m}<\psi_m<2\log|F|+\frac{2t_0}{3m}$ holds, then $-\frac{t_0}{m}\le\Psi_m\le0$, and
\[(1+\delta)\max\{\psi_m,2\log|F|\}\le\varphi_m\le (1+\delta)\left(\psi_m+\frac{t_0}{m}\right).\]

Thus, we have
\[\{\Psi_m<-t_0\}\subset\{\psi-2\log|F|<-t_0\}\subset\{\Psi<-t_0\},\]
and $\Psi_m\le0$ on $D\backslash\{F=0\}$.
\end{proof}

Now we begin to prove Lemma \ref{l:L2}.

Note that $D\backslash\{F=0\}$ is a pseudoconvex domain. The following remark shows that it suffices to consider the case of Lemma \ref{l:L2} that $F(z)\not=0$ for any $z\in D$.

\begin{Remark}
	Assume that there exists a holomorphic function $F_1$ on $D\backslash\{F=0\}$ such that
	\begin{equation}
		\label{eq:0212a}\begin{split}
		&\int_{D\backslash\{F=0\}}|F_1-(1-b_{t_0,B}(\Psi))fF^{1+\delta}|^2e^{-\varphi+v_{t_0,B}(\Psi)-\Psi}\\
		\le&\left(\frac{1}{\delta}+1-e^{-t_0-B} \right)\int_D\frac{1}{B}\mathbb{I}_{\{-t_0-B<\Psi<-t_0\}}|f|^2e^{-\Psi}.
	\end{split}	
	\end{equation}
	Note that $v_{t_0,B}(\Psi)-\Psi\ge0$ on $D$ and $|F|^{2(1+\delta)}e^{-\varphi}=1$ on $\{\Psi<-t_0\}$. Then it follows from $\int_{\{\Psi<-t_0\}}|f|^2<+\infty$ and inequality \eqref{eq:0212a} that
	\begin{displaymath}
		\begin{split}
			&\int_{D\backslash\{F=0\}}|F_1|^2e^{-\varphi}\\
			\le&2\int_{D\backslash\{F=0\}}|(1-b_{t_0,B}(\Psi))f|^2+2\int_{D\backslash\{F=0\}}|F_1-(1-b_{t_0,B}(\Psi))fF^{1+\delta}|^2e^{-\varphi}\\
			\le&2\int_{\{\Psi<-t_0\}}|f|^2+\left(\frac{1}{\delta}+1-e^{-t_0-B} \right)\int_D\frac{1}{B}\mathbb{I}_{\{-t_0-B<\Psi<-t_0\}}|f|^2e^{-\Psi}\\
			<&+\infty,
		\end{split}
	\end{displaymath}
	which implies that there exists a holomorphic function $\tilde F$ on $D$ such that $\tilde F=F_1$ on $D\backslash\{F=0\}$. Thus, we have
		\begin{displaymath}\begin{split}
		&\int_{D}|\tilde F-(1-b_{t_0,B}(\Psi))fF^{1+\delta}|^2e^{-\varphi+v_{t_0,B}(\Psi)-\Psi}\\
		\le&\left(\frac{1}{\delta}+1-e^{-t_0-B} \right)\int_D\frac{1}{B}\mathbb{I}_{\{-t_0-B<\Psi<-t_0\}}|f|^2e^{-\Psi}.	
	\end{split}	\end{displaymath}
\end{Remark}

Let $D_1\Subset\ldots\Subset D_j\Subset D_{j+1}\Subset\ldots$ be a sequence of strongly pseudoconvex domains such that $\bigcup_{j=1}^{+\infty}D_j=D$. As discussed above, there exist smooth plurisubharmonic functions $\{\varphi_m\}_{m\in\mathbb{Z}_{>0}}$ and smooth functions $\{\Psi_m\}_{m\in\mathbb{Z}_{>0}}$ on $D$, such that:
	
	(1) $\{\varphi_m+\Psi_m\}_{m\in\mathbb{Z}_{>0}}$ and $\{\varphi_m+(1+\delta)\Psi_m\}_{m\in\mathbb{Z}_{>0}}$ are smooth plurisubharmonic functions;
	
	(2) the sequence $\{\varphi_m\}_{m\in\mathbb{Z}_{>0}}$  is  convergent to $\varphi$, and there exists a smooth plurisubharmonic function $\varphi_0$ on $D$ such that $\varphi_0\ge\varphi_m\ge\varphi$ on $D$ for any $m\in\mathbb{Z}_{>0}$;
	
	(3) the sequence $\{\varphi_m+(1+\delta)\Psi_m\}_{m\in\mathbb{Z}_{>0}}$ is decreasingly convergent to $\varphi+(1+\delta)\Psi$;
	
	$(4)$ $\{\Psi_m<-t_0\}\subset\{\Psi<-t_0\}$  and $\Psi_m\le0$ on $D$ for any $m\in\mathbb{Z}_{>0}$.
	
\

In the following, we prove Lemma \ref{l:L2} in $6$ steps.

  \

\emph{Step 1: Some Notations}

\

Let $\epsilon\in(0,\frac{1}{8}B)$.
Let $\{v_{\epsilon}\}_{\epsilon\in(0,\frac{1}{8}B)}$ be a family of smooth increasing convex functions on $\mathbb{R}$,
which are continuous functions on $\mathbb{R}\cup\{-\infty\}$, such that:

 $1)$ $v_{\epsilon}(t)=t$ for $t\geq-t_{0}-\epsilon$, $v_{\epsilon}(t)=constant$ for $t<-t_{0}-B+\epsilon$ and are pointwise convergent to $v_{t_0,B}$, when $\epsilon\to 0$;

 $2)$ $v''_{\epsilon}(t)$ are pointwise convergent to $\frac{1}{B}\mathbb{I}_{(-t_{0}-B,-t_{0})}$, when $\epsilon\to 0$,
 and $0\leq v''_{\epsilon}(t)\leq \frac{2}{B}\mathbb{I}_{(-t_{0}-B+\epsilon,-t_{0}-\epsilon)}$ for any $t\in \mathbb{R}$;

 $3)$ $v'_{\epsilon}(t)$ are pointwise convergent to $b_{t_0,B}(t)$ which is a continuous function on $\mathbb{R}\cup\{-\infty\}$, when $\epsilon\to 0$, and $0\leq v'_{\epsilon}(t)\leq1$ for any $t\in \mathbb{R}$.

One can construct the family $\{v_{\epsilon}\}_{\epsilon\in(0,\frac{1}{8}B)}$ by the setting
\begin{equation}
\label{eq:15}
v_{\epsilon}(t):=\int_{0}^{t}\left(\int_{-\infty}^{t_{1}}\left(\frac{1}{B-4\epsilon}
\mathbb{I}_{(-t_{0}-B+2\epsilon,-t_{0}-2\epsilon)}*\rho_{\frac{1}{4}\epsilon}\right)(s)ds\right)dt_{1},
\end{equation}
where $\rho_{\frac{1}{4}\epsilon}$ is the kernel of convolution satisfying $\mathrm{supp}(\rho_{\frac{1}{4}\epsilon})\subset (-\frac{1}{4}\epsilon,\frac{1}{4}\epsilon)$.
Then it follows that
$$v''_{\epsilon}(t)=\frac{1}{B-4\epsilon}\mathbb{I}_{(-t_{0}-B+2\epsilon,-t_{0}-2\epsilon)}*\rho_{\frac{1}{4}\epsilon}(t),$$
and
$$v'_{\epsilon}(t)=\int_{-\infty}^{t}\left(\frac{1}{B-4\epsilon}\mathbb{I}_{(-t_{0}-B+2\epsilon,-t_{0}-2\epsilon)}
*\rho_{\frac{1}{4}\epsilon}\right)(s)ds.$$
It is clear that $\lim_{\epsilon\to 0}v_{\epsilon}(t)=v_{t_0,B}(t)$, and $\lim_{\epsilon\to 0}v'_{\epsilon}(t)=b_{t_0,B}(t)$.

 Let $\eta=s(-v_{\epsilon}(\Psi_m))$ and $\phi=u(-v_{\epsilon}(\Psi_m))$,
where $s\in C^{\infty}((0,+\infty))$ satisfies $s\geq\frac{1}{\delta}$ and $s'>0$, and
$u\in C^{\infty}([0,+\infty))$, such that $u''s-s''>0$ and $s'-u's=1$ on $(0,+\infty)$.
Let $\Phi=\phi+\varphi_{m}+\Psi_m$.

\

\emph{Step 2: Solving the $\bar\partial$-equation}

\

Now let $\alpha=\sum_{j=1}^n\alpha_{\bar j}d\bar z^j\in$Dom$_{D_j}(\bar\partial^*)\cap $Ker$(\bar\partial)\cap C_{(0,1)}^{\infty}(\overline{D_j})$. It follows from Cauchy-Schwarz's inequality that
\begin{equation}
\label{eq:17}
		2\mathrm{Re}(\bar\partial_{\Phi}^*\alpha,\alpha\llcorner(\bar\partial\eta)^{\sharp})_{D_j,\Phi}\geq -\int_{D_j}g^{-1}|\bar\partial_{\Phi}^*\alpha|^2e^{-\Phi}+\sum_{j,k=1}^n\int_{D_j}-g(\partial_j\eta)(\bar\partial_k\eta)\alpha_{\bar j}\bar\alpha_{\bar k}e^{-\Phi},
\end{equation}
 where $g$ is a positive continuous function on $\overline{D_j}$.
Using Lemma \ref{l23} and inequality \eqref{eq:17}, since  $D_j$ is a strongly pseudoconvex domain, we get
\begin{equation}
	\label{eq:18}
	\begin{split}
		&\int_{D_j}(\eta+g^{-1})|\bar\partial_{\Phi}^*\alpha|^2e^{-\Phi}\\
		\geq &\sum_{j,k=1}^n\int_{D_j}(-\partial_j\bar\partial_k\eta+\eta\partial_j\bar\partial_k\Phi-g(\partial_j\eta)(\bar\partial_k\eta))\alpha_{\bar j}\bar\alpha_{\bar k}e^{-\Phi}\\
		= &\sum_{j,k=1}^n\int_{D_j}(-\partial_j\bar\partial_k\eta+\eta\partial_j\bar\partial_k(\phi+\varphi_m+\Psi_m)-g(\partial_j\eta)(\bar\partial_k\eta))\alpha_{\bar j}\bar\alpha_{\bar k}e^{-\Phi}.	\end{split}
\end{equation}
We need some calculations to determine $g$.

Since
\begin{equation*}
	\partial_j\bar\partial_k\eta=-s'(-v_{\epsilon}(\Psi_m))\partial_j\bar\partial_k(v_{\epsilon}(\Psi_m))+s''(-v_{\epsilon}(\Psi_m))\partial_jv_{\epsilon}(\Psi_m)\bar\partial_kv_{\epsilon}(\Psi_m),
	\end{equation*}
and
\begin{equation*}
	\partial_j\bar\partial_k\phi=-u'(-v_{\epsilon}(\Psi_m))\partial_j\bar\partial_k(v_{\epsilon}(\Psi_m))+u''(-v_{\epsilon}(\psi_m))\partial_jv_{\epsilon}(\Psi_m)\bar\partial_kv_{\epsilon}(\Psi_m),
\end{equation*}
for any $j,k$ $(1\leq j,k\leq n)$, we have
\begin{equation}
\label{eq:21}
\begin{split}
&\sum_{j,k=1}^n(-\partial_j\bar{\partial}_k\eta+\eta\partial_j\bar{\partial}_k\phi-g(\partial_j\eta)(\bar\partial_k\eta))\alpha_{\bar j}\bar\alpha_{\bar k}
\\=&(s'-su')\sum_{j,k=1}^n\partial_j\bar{\partial_k}v_{\epsilon}(\Psi_m)\alpha_{\bar j}\bar\alpha_{\bar k}\\
&+((u''s-s'')-gs'^{2})\sum_{j,k=1}^n\partial_j
(-v_{\epsilon}(\Psi_m))\bar{\partial}_k(-v_{\epsilon}(\Psi_m))\alpha_{\bar j}\bar\alpha_{\bar k}
\\=&(s'-su')\sum_{j,k=1}^n(v'_{\epsilon}(\Psi_m)\partial_j\bar{\partial}_k\Psi_m+v''_{\epsilon}(\Psi_m)
\partial_j(\Psi_m)\bar{\partial}_k(\Psi_m))\alpha_{\bar j}\bar\alpha_{\bar k}
\\&+((u''s-s'')-gs'^{2})\sum_{j,k=1}^n\partial_j
(-v_{\epsilon}(\Psi_m))\bar{\partial}_k(-v_{\epsilon}(\Psi_m))\alpha_{\bar j}\bar\alpha_{\bar k}
.
\end{split}
\end{equation}
We omit the composite item $-v_{\epsilon}(\Psi_m)$ after $s'-su'$ and $(u''s-s'')-gs'^{2}$ in the above equalities. As $\varphi_m+\Psi_m$ and $\varphi_m+(1+\delta)\Psi_m$ are plurisubharmonic functions and $v'_{\epsilon}\in[0,1]$, we have
$$(1-v'_{\epsilon}(\Psi_m))\sqrt{-1} \partial\bar\partial(\varphi_m+\Psi_m)+v'_{\epsilon}(\Psi_m)\sqrt{-1}\partial\bar\partial(\varphi_m+(1+\delta)\Psi_m)\ge0$$
on $\overline{D_j}$,
which implies that
\begin{equation}
	\label{eq:19}\begin{split}
	&\eta\sqrt{-1} \partial\bar\partial(\varphi_m+\Psi_m)+v'_{\epsilon}(\Psi_m)\sqrt{-1} \partial\bar\partial\Psi_m\\
	\ge&\frac{1}{\delta}\sqrt{-1} \partial\bar\partial(\varphi_m+\Psi_m)+v'_{\epsilon}(\Psi_m)\sqrt{-1} \partial\bar\partial\Psi_m\ge 0
	\end{split}
\end{equation}
on $\overline{D_j}$.

Let $g=\frac{u''s-s''}{s'^2}(-v_{\epsilon}(\Psi_m))$. It follows that $\eta+g^{-1}=\left(s+\frac{s'^2}{u''s-s''}\right)(-v_{\epsilon}(\Psi_m))$.
As $s'-su'=1$, using inequality \eqref{eq:18}, inequality \eqref{eq:21} and inequality \eqref{eq:19}, we obtain that
\begin{equation}
	\label{eq:22}\begin{split}
\int_{D_j}(\eta+g^{-1})|\bar\partial_{\Phi}^*\alpha|^2e^{-\Phi}
	\ge \int_{D_j} v''_{\epsilon}(\Psi_m)|\alpha\llcorner(\bar\partial\Psi_m)^{\sharp}|^2e^{-\Phi}.
	\end{split}
\end{equation}

As $fF^{1+\delta}$ is holomorphic on $\{\Psi<-t_0\}$ and $\overline{\{\Psi_m<-t_0-\epsilon\}}\subset\{\Psi_m<-t_0\} \subset\{\Psi<-t_0\}$,
\[\lambda:=\bar\partial\Big(\big(1-v'_{\epsilon}(\Psi_m)\big)fF^{1+\delta}\Big)\]
is well-defined and smooth on $D$. By the definition of contraction, Cauchy-Schwarz's inequality and inequality \eqref{eq:22}, it follows that
\begin{equation}
	\label{eq:23}
	\begin{split}
		|(\lambda,\alpha)_{D_j,\Phi}|^2=&|(v''_{\epsilon}(\Psi_m)fF^{1+\delta}\bar\partial\Psi_m,\alpha)_{D_j,\Phi}|^2\\
		=&|(v''_{\epsilon}(\Psi_m)fF^{1+\delta},\alpha\llcorner(\bar\partial\Psi_m)^{\sharp})_{D_j,\Phi}|^2\\
		\leq&\left(\int_{D_j}v''_{\epsilon}(\Psi_m)|fF^{1+\delta}|^2e^{-\Phi}\right)\left(\int_{D_j}v''_{\epsilon}(\Psi_m)|\alpha\llcorner(\bar\partial\Psi_m)^{\sharp}|^2e^{-\Phi}\right)\\
		\leq&\left(\int_{D_j}v''_{\epsilon}(\Psi_m)|fF^{1+\delta}|^2e^{-\Phi}\right)\left(\int_{D_j}(\eta+g^{-1})|\bar\partial_{\Phi}^*\alpha|^2e^{-\Phi}\right).
	\end{split}
\end{equation}
Set $\mu:=(\eta+g^{-1})^{-1}$. By Lemma \ref{l24}, there exists a locally $L^1$ function $u_{m,\epsilon,j}$ on $D_j$ such that $\bar\partial u_{m,\epsilon,j}=\lambda$, and
\begin{equation}
	\label{eq:24}
		\int_{D_j}|u_{m,\epsilon,j}|^2(\eta+g^{-1})^{-1}e^{-\Phi}\leq\int_{D_j}v''_{\epsilon}(\Psi_m)|fF^{1+\delta}|^2e^{-\Phi}.
\end{equation}

Assume that we can choose $\eta$ and $\phi$ such that $e^{v_{\epsilon}(\Psi_m)}e^{\phi}=(\eta+g^{-1})^{-1}$.
Then inequality \eqref{eq:24} becomes
\begin{equation}
 \label{eq:25}
 \begin{split}
 &\int_{D_j}|u_{m,\epsilon,j}|^{2}e^{v_{\epsilon}(\Psi_m)-\varphi_{m}-\Psi_m}
  \leq\int_{D_j}v''_{\epsilon}(\Psi_m)| F|^2e^{-\phi-\varphi_{m}-\Psi_m}.
  \end{split}
\end{equation}
Let
\[F_{m,\epsilon,j}:=-u_{m,\epsilon,j}+\big(1-v'_{\epsilon}(\Psi_m)\big){fF^{1+\delta}}.\]
Then $F_{m,\epsilon,j}$ is holomorphic on $D_j$. Now inequality \eqref{eq:25} becomes
\begin{equation}
 \label{eq:26}
 \begin{split}
 &\int_{D_j}|F_{m,\epsilon,j}-(1-v'_{\epsilon}(\Psi_m)){fF^{1+\delta}}|^{2}e^{v_{\epsilon}(\Psi_m)-\varphi_{m}-\Psi_m}
  \\
  \leq&\int_{D_j}(v''_{\epsilon}(\Psi_m))|fF^{1+\delta}|^2e^{-\phi-\varphi_{m}-\Psi_m}.
  \end{split}
\end{equation}

\

\emph{Step 3: $m\rightarrow+\infty$}

\

 Note that
 \[\sup_{m}\sup_{D_j}e^{-\phi}=\sup_{m}\sup_{D_j}e^{-u(-v_{\epsilon}(\Psi_m))}<+\infty.\]
As $\{\Psi_m<-t_0-\epsilon\}\subset\{\Psi<-t_0\}$ and $|F^{1+\delta}|^2e^{-\varphi}=1$ on $\{\Psi<-t_0\}$,
we have
\begin{displaymath}
	\begin{split}
		(v''_{\epsilon}(\Psi_m))|fF^{1+\delta}|^2e^{-\phi-\varphi_{m}-\Psi_m}&\le e^{t_0+B}(v''_{\epsilon}(\Psi_m))|fF^{1+\delta}|^2e^{-\phi-\varphi}\\
		&\le \left(\sup_{m}\sup_{D_j}e^{-\phi}\right)e^{t_0+B}\mathbb{I}_{\{\Psi<-t_0\}}|f|^2
	\end{split}
\end{displaymath}
on $D_j$. Then it follows from $\int_{\{\Psi<-t_0\}}|f|^2<+\infty$ and the dominated convergence theorem that
\begin{equation}
\label{eq:27}
\begin{split}
 &\lim_{m\rightarrow +\infty}\int_{D_j}(v''_{\epsilon}(\Psi_m))|fF^{1+\delta}|^2e^{-\phi-\varphi_{m}-\Psi_m}\\
 =&
\int_{D_j}(v''_{\epsilon}(\Psi))|fF^{1+\delta}|^2e^{-u(-v_{\epsilon}(\Psi))-\varphi-\Psi}<+\infty.
\end{split}
\end{equation}
Since $\inf_{m}\inf_{D_j}e^{v_{\epsilon}(\Psi_m)-\varphi_{m}-\Psi_m}\ge\inf_{D_j}e^{-\varphi_{0}}>0$, by inequality \eqref{eq:26}, we have
\[\sup_{m}\int_{D_j}|F_{m,\epsilon,j}-(1-v'_{\epsilon}(\Psi_m)){fF^{1+\delta}}|^{2}<+\infty.\]
According to
\begin{equation*}
\left|(1-v'_{\epsilon}(\Psi_m)){fF^{1+\delta}}\right|\leq \left(\sup_{D_j}|F|^{1+\delta}\right)\mathbb{I}_{\{\Psi<-t_{0}\}}|f|,
\end{equation*}
and $\int_{\{\Psi<-t_0\}}|f|^2<+\infty$, we get $\sup_{m}\int_{D_j}|F_{m,\epsilon,j}|^{2}<+\infty$,
which implies that there exists a subsequence of $\{F_{m,\epsilon,j}\}_{m\in \mathbb{Z}_{>0}}$
(also denoted by $F_{m,\epsilon,j}$) compactly convergent to a holomorphic function $F_{\epsilon,j}$ on $D_j$.
Now Fatou's Lemma, inequality \eqref{eq:26} and inequality \eqref{eq:27} indicate that
\begin{equation}
 \label{eq:31}
 \begin{split}
 &\int_{D_j}|F_{\epsilon,j}-(1-v'_{\epsilon}(\Psi)){fF^{1+\delta}}|^{2}e^{v_{\epsilon}(\Psi)-\varphi-\Psi}\\
 \le&\liminf_{m\rightarrow+\infty}\int_{D_j}\left|F_{m,\epsilon,j}-(1-v'_{\epsilon}(\Psi_m)){fF^{1+\delta}}\right|^{2}e^{v_{\epsilon}(\Psi_m)-\varphi_{m}-\Psi_m}
  \\\le&\lim_{m\rightarrow+\infty}\int_{D_j}v''_{\epsilon}(\Psi_m)|fF^{1+\delta}|^2e^{-\phi-\varphi_{m}-\Psi_m}\\
 =&
\int_{D_j}v''_{\epsilon}(\Psi)|fF^{1+\delta}|^2e^{-u(-v_{\epsilon}(\Psi))-\varphi-\Psi}\\
=&\int_{D_j}v''_{\epsilon}(\Psi)|f|^2e^{-u(-v_{\epsilon}(\Psi))-\Psi}.
  \end{split}
\end{equation}

\

\emph{Step 4: $\epsilon\rightarrow0+0$}

\

Note that $\sup_{\epsilon}\sup_{D_j}e^{-u(-v_{\epsilon}(\Psi))}<+\infty$, and
\begin{displaymath}
(v''_{\epsilon}(\Psi))|f|^2e^{-u(-v_{\epsilon}(\Psi))-\Psi}
		\le\left(\sup_{\epsilon}\sup_{D_j}e^{-u(-v_{\epsilon}(\Psi))}\right)e^{t_0+B}\mathbb I_{\{\Psi<-t_0\}}|f|^2
\end{displaymath}
on $D_j$. Then it follows from $\int_{\{\Psi<-t_0\}}|f|^2<+\infty$ and the dominated convergence theorem that
\begin{equation}
\label{eq:32}
\begin{split}
&\lim_{\epsilon\to0}\int_{D_j}(v''_{\epsilon}(\Psi))|f|^2e^{-u(-v_{\epsilon}(\Psi))-\Psi}
\\=&\int_{D_j}\frac{1}{B}\mathbb{I}_{\{-t_{0}-B<\Psi<-t_{0}\}}|f|^2e^{-u(-v_{t_0,B}(\Psi))-\Psi}
\\\leq&\left(\sup_{D_j}e^{-u(-v_{t_0,B}(\Psi))}\right)\int_{D_j}\frac{1}{B}\mathbb{I}_{\{-t_{0}-B<\Psi<-t_{0}\}}|f|^2e^{-\Psi}\\
<&+\infty.
\end{split}
\end{equation}
Since $\inf_{\epsilon}\inf_{D_j}e^{v_{\epsilon}(\Psi)-\varphi-\Psi}>0$, according to inequality \eqref{eq:31} and \eqref{eq:32}, we have
\[\int_{D_j}\left|F_{\epsilon,j}-(1-v'_{\epsilon}(\Psi)){fF^{1+\delta}}\right|^{2}<+\infty.\]
Combining with
\begin{equation*}
\sup_{\epsilon}\int_{D_j}|(1-v'_{\epsilon}(\Psi)){fF^{1+\delta}}|^{2}\leq\left(\sup_{D_j}|F|^{2(1+\delta)}\right)\int_{D_j}\mathbb{I}_{\{\Psi<-t_{0}\}}|f|^{2}<+\infty,
\end{equation*}
one can obtain $\sup_{\epsilon}\int_{D_j}|F_{\epsilon,j}|^{2}<+\infty$, which implies that there exists a subsequence of $\{F_{\epsilon,j}\}_{\epsilon\rightarrow0+0}$ (denoted by $\{F_{\epsilon_l,j}\}_{l\in\mathbb{Z}_{>0}}$)
compactly convergent to a holomorphic function $F_{j}$ on $D_j$. Now Fatou's Lemma, inequality \eqref{eq:31} and inequality \eqref{eq:32} verify that
\begin{equation}
 \label{eq:36}
 \begin{split}
 &\int_{D_j}|F_{j}-(1-b_{t_0,B}(\Psi)){fF^{1+\delta}}|^{2}e^{v_{t_0,B}(\Psi)-\varphi-\Psi}\\
 \le&\liminf_{l \rightarrow+\infty}\int_{D_j}|F_{\epsilon_l,j}-(1-v'_{\epsilon_l}(\Psi)){fF^{1+\delta}}|^{2}e^{v_{\epsilon_l}(\Psi)-\varphi-\Psi}
  \\\le&\liminf_{l \rightarrow +\infty}\int_{D_j}(v''_{\epsilon_l}(\Psi))|f|^2e^{-u(-v_{\epsilon_l}(\Psi))-\Psi}\\\leq&\left(\sup_{D_j}e^{-u(-v_{t_0,B}(\Psi))}\right)\int_{D_j}\frac{1}{B}\mathbb{I}_{\{-t_{0}-B<\Psi<-t_{0}\}}|f|^2e^{-\Psi}.
  \end{split}
\end{equation}

\

\emph{Step 5: $j\rightarrow+\infty$}

\

Note that
\begin{equation}
\label{eq:37}
\begin{split}
	&\lim_{j\rightarrow+\infty}\left(\sup_{D_j}e^{-u(-v_{t_0,B}(\Psi))}\right)\int_{D_j}\frac{1}{B}\mathbb{I}_{\{-t_{0}-B<\Psi<-t_{0}\}}|f|^2e^{-\Psi}\\
	\le&\left(\sup_{D}e^{-u(-v_{t_0,B}(\Psi))}\right)\int_{D}\frac{1}{B}\mathbb{I}_{\{-t_{0}-B<\Psi<-t_{0}\}}|f|^2e^{-\Psi}<+\infty,
\end{split}
\end{equation}
and $v_{t_0,B}(\Psi)-\Psi\ge0$. Then it follows from inequality \eqref{eq:36} that
$$\limsup_{j\rightarrow+\infty}\int_{D_j}|F_{j}-(1-b_{t_0,B}(\Psi)){fF^{1+\delta}}|^{2}e^{-\varphi}<+\infty.$$
Combining with
\[\int_{D_j}|(1-b_{t_0,B}(\Psi)){fF^{1+\delta}}|^2e^{-\varphi}\le\int_{\{\Psi<-t_0\}}|f|^2<+\infty,\]
we obtain
$$\limsup_{j\rightarrow+\infty}\int_{D_j}|F_{j}|^{2}e^{-\varphi}<+\infty.$$
Then we can extract a compactly convergent subsequence of $\{F_{j}\}$ (also denoted by $\{F_{j}\}$),
which is convergent to a holomorphic function $\tilde{F}$ on $D$. It follows from Fatou's Lemma, inequality \eqref{eq:36} and inequality \eqref{eq:37} that
\begin{equation}
 \label{eq:41}
 \begin{split}
 &\int_{D}|\tilde F-(1-b_{t_0,B}(\Psi)){fF^{1+\delta}}|^{2}e^{v_{t_0,B}(\Psi)-\varphi-\Psi}\\
 \le&\liminf_{j \rightarrow+\infty}\int_{D_j}|F_{j}-(1-b_{t_0,B}(\Psi)){fF^{1+\delta}}|^{2}e^{v_{t_0,B}(\Psi)-\varphi-\Psi}
  \\\le&\lim_{j\rightarrow+\infty}\left(\sup_{D_j}e^{-u(-v_{t_0,B}(\Psi))}\right)\int_{D_j}\frac{1}{B}\mathbb{I}_{\{-t_{0}-B<\Psi<-t_{0}\}}|f|^2e^{-\Psi}\\
	\le&\left(\sup_{D}e^{-u(-v_{t_0,B}(\Psi))}\right)\int_{D}\frac{1}{B}\mathbb{I}_{\{-t_{0}-B<\Psi<-t_{0}\}}|f|^2e^{-\Psi}.
  \end{split}
\end{equation}

\

\emph{Step 6: ODE system}

\

Finally, it suffices to find $\eta$ and $\phi$ such that
$\eta+g^{-1}=e^{-v_{\epsilon}(\Psi_m)}e^{-\phi}$  on $D_j$ and $s'-u's=1$. As $\eta=s(-v_{\epsilon}(\Psi_m))$ and $\phi=u(-v_{\epsilon}(\Psi_m))$, we have
\[(\eta+g^{-1}) e^{v_{\epsilon}(\Psi_m)}e^{\phi}=\left(\left(s+\frac{s'^{2}}{u''s-s''}\right)e^{-t}e^{u}\right)\circ\big(-v_{\epsilon}(\Psi_m)\big).\]

Summarizing the above discussions about $s$ and $u$, we are naturally led to a system of ODEs (see \cite{guan-zhou12,guan-zhou13p,guan-zhou13ap,GZeff}):
\begin{equation}
\label{eq:42}
\begin{split}
&(1)\,\,\left(s+\frac{s'^{2}}{u''s-s''}\right)e^{u-t}=1, \\
&(2)\,\,s'-su'=1,
\end{split}
\end{equation}
where $t\in(0,+\infty)$. Solving the ODE system \eqref{eq:42}, we get
\[u(t)=-\log\left(\frac{1}{\delta}+1-e^{-t}\right),\]
and
\[s(t)=\frac{\left(1+\frac{1}{\delta}\right)t+\frac{1}{\delta}\left(1+\frac{1}{\delta}\right)}{\frac{1}{\delta}+1-e^{-t}}-1,\]
(see \cite{GZeff}).
It follows that $s\in C^{\infty}((0,+\infty))$ satisfying $s\ge\frac{1}{\delta}$ and $s'>0$, and
$u\in C^{\infty}([0,+\infty))$ satisfying $u''s-s''>0$.

As $u(t)=-\log\left(\frac{1}{\delta}+1-e^{-t}\right)$ is decreasing with respect to $t$, and $0\geq v(t)\geq\max\{t,-t_{0}-B_{0}\}\geq -t_{0}-B_{0}$ for any $t\leq0$, we can see that
\begin{equation*}
\begin{split}
\sup_{D}e^{-u(-v_{t_0,B}(\Psi))}
\leq\sup_{t\in[0,t_{0}+B]}e^{-u(t)}
=\frac{1}{\delta}+1-e^{-t_0-B}.
\end{split}
\end{equation*}
Therefore, we are done.

Eventually, we complete the proof of Lemma \ref{l:L2}.

\subsection{An example on the Hartogs triangle}\label{sec:eg}

\

In the second part of the appendix, for the sake of the completeness, we drag an explicit example for the objects described in the present paper. In this example, the function $F$ is not trivial (i.e. not a constant function). A similar example can also be seen in \cite{BG22}.

Let $D=\Delta^2$ be the unit polydisc in $\mathbb{C}^2$ with the coordinate $(z,w)$, $\psi=2\log|z|$ a plurisubharmonic function on $\Delta^2$, and $F=w^k$ a holomorphic function on $\Delta^2$ for a positive integer $k$. Then we have
\[\Psi(z,w):=\min\{\psi-2\log|F|, 0\}=\min\left\{2\log\left|\frac{z}{w^k}\right|, 0\right\},\]
where $\Psi(o,w):=0$ for any $w\in\Delta$. It follows that
\[D_t:=\{\Psi<-t\}=\left\{(z,w)\in\Delta\times\Delta : |z|<e^{-t/2}|w|^k\right\}, \ \forall t\in [0,+\infty).\]
Especially,
\[D_0=\left\{(z,w)\in\mathbb{C}^2 : |z|<|w|^k<1\right\},\]
which is also called the (generalized) \emph{Hartogs triangle} (see \cite{EM16,Sha15}). One can see that $o$ is a boundary point of $D_t$ for any $t\ge 0$.

Fix $t\ge 0$. Note that $D_t$ is a Reinhardt domain (which means that $D_t$ is invariant under the $\mathbb{T}^n-$action). Then any holomorphic function on $D_t$ can be written as its Laurent expansion. Let
\[f(z,w)=\sum_{(j,l)\in\mathbb{Z}^2}c_{j,l}z^jw^l\in\mathcal{O}(D_t).\]
One can verify that
\[\int_{D_t}|f|^2=\sum_{(j,l)\in\mathbb{Z}^2}|c_{j,l}|^2\int_{D_t}|z|^{2j}|w|^{2l},\]
where
\begin{equation*}
	\int_{D_t}|z|^{2j}|w|^{2l}=\left\{
		\begin{array}{ll}
			\dfrac{\pi^2 e^{-(j+1)t}}{(j+1)(jk+k+l+1)}, & \text{if} \ j\geq 0, \ jk+k+l\geq 0,\\
			+\infty, & \text{otherwise.}
		\end{array}\right.
\end{equation*}
Similarly,
\[\int_{D_t}|f|^2e^{-\Psi}=\sum_{(j,l)\in\mathbb{Z}^2}|c_{j,l}|^2\int_{D_t}|z|^{2j-2}|w|^{2l+2k},\]
where
\begin{equation*}
	\int_{D_t}|z|^{2j-2}|w|^{2l+2k}=\left\{
		\begin{array}{ll}
			\dfrac{\pi^2 e^{-jt}}{j(jk+k+l+1)}, & \text{if} \ j\geq 1, \ jk+k+l\geq 0,\\
			+\infty, & \text{otherwise.}
		\end{array}\right.
\end{equation*}
Then it can be summarized that
\[I(\Psi)_o=\Big\{h(z,w)=\sum_{(j,l)\in\mathbb{Z}^2}a_{j,l}z^jw^l \text{\ near \ } o : a_{j,l}=0, \ \forall (j,l)\in \mathcal{S} \Big\},\]
where
\[\mathcal{S}=\big\{(j,l)\in\mathbb{Z}^2 : j\le 0 \ \text{or} \ l\leq -(jk+k+1)\big\}.\]

Set $J=I(\Psi)_o$, and $f=w^{-k}$. It is not hard to calculate that
\[G(t;\Psi,J,f)=\int_{D_t}|f|^2=\pi^2 e^{-t},\]
which shows that $G(-\log r; \Psi,J,f)$ is concave with respect to $r\in (0,1]$ in this case. 

Under the same assumptions above, one can also verify that $c_o^{fF}(\psi)=1/2$, and
\begin{equation*}
	\frac{1}{r^2}\int_{\{c_o^{fF}(\psi)\psi-\log|F|<\log r\}}|f|^2=\frac{1}{r^2}\int_{D_{-2\log r}}|w|^{-2k}=\pi^2.
\end{equation*}
On the other hand, since
\[\Psi_1:=\min\{2c_o^{fF}(\psi)\psi-2\log|F|,0\}=\Psi,\]
with some similar computations for $a\Psi$ ($a>1$), one can see that
\[I_+(\Psi)_o=\bigcup_{a>1}I(a\Psi)_o=I(\Psi)_o.\]
Now according to the value of $G(t;\Psi,J,f)$, it can be obtained that
\[\frac{1}{r^2}\int_{\{c_o^{fF}(\psi)\psi-\log|F|<\log r\}}|f|^2=\pi^2=G(0;\Psi_1,I_+(\Psi)_o,f),\]
for any $r\in (0,1]$. This also shows that the effectiveness result in Theorem \ref{thm:J_M} is sharp.

It is easy to check that $a^f_o(\Psi)=1/2$, thus one can verify that the lower bound in Theorem \ref{thm:J-M2} is also sharp. The strong openness property for $\Psi$ has been checked.

\bibliographystyle{references}
\bibliography{xbib}

\begin{thebibliography}{100}

\bibitem{BG22}S.J. Bao and Q.A. Guan, Modules at boundary points, fiberwise Bergman kernels, and log-subharmonicity II -- on Stein manifolds, Preprint, arXiv:2205.08044v2.

\bibitem{berndtsson}B. Berndtsson, The extension theorem of Ohsawa-Takegoshi and the theorem of
Donnelly-Fefferman, Ann. L'Inst. Fourier (Grenoble) 46 (1996), no.4, 1083--1094.

\bibitem{Berndtsson2}B. Berndtsson,
The openness conjecture for plurisubharmonic functions, arXiv:1305.5781.

\bibitem{berndtsson20}B. Berndtsson, Lelong numbers and vector bundles, J. Geom. Anal. 30 (2020), no. 3, 2361-2376.




\bibitem{cao17}J.Y. Cao, Ohsawa-Takegoshi extension theorem for compact K{\"a}hler manifolds and applications, Complex and symplectic geometry, 19-38,
Springer INdAM Ser., 21, Springer, Cham, 2017.

\bibitem{cdM17}J.Y. Cao, J.-P. Demailly and S. Matsumura, A general extension theorem for cohomology classes on non reduced analytic subspaces, Sci. China Math. 60 (2017), no. 6, 949-962, DOI 10.1007/s11425-017-9066-0.


\bibitem{DEL18}T. Darvas, E. Di Nezza and H.C. Lu,
Monotonicity of nonpluripolar products and complex Monge-Amp{\'e}re equations with prescribed singularity,
Anal. PDE 11 (2018), no. 8, 2049-2087.

\bibitem{DEL21}T. Darvas, E. Di Nezza and H.C. Lu, The metric geometry of singularity types,
J. Reine Angew. Math. 771 (2021), 137-170.





\bibitem{Demaillybook}J.-P. Demailly,
Complex analytic and differential geometry,
electronically accessible at \href{https://www-fourier.ujf-grenoble.fr/~demailly/documents.html}{https://www-fourier.ujf-grenoble.fr/\textasciitilde demailly/documents.html}.



\bibitem{DemaillyAG}J.-P. Demailly,
Analytic Methods in Algebraic Geometry,
Higher Education Press, Beijing, 2010.


\bibitem{DemaillySoc}J.-P. Demailly,
Multiplier ideal sheaves and analytic methods in algebraic geometry, School on Vanishing Theorems and Effective Result in Algebraic Geometry (Trieste,2000),1-148,ICTP lECT.Notes, 6, Abdus Salam Int. Cent. Theoret. Phys., Trieste, 2001.

\bibitem{DEL}J.-P. Demailly, L.Ein and R.Lazarsfeld,
A subadditivity property of multiplier ideals,
Michigan Math. J. 48 (2000) 137-156.

\bibitem{DK01}J.-P. Demailly and J. Koll\'ar,
Semi-continuity of complex singularity exponents and K\"ahler-Einstein metrics on Fano orbifolds,
Ann. Sci. \'Ec. Norm. Sup\'er. (4) 34 (4) (2001) 525-556.

\bibitem{DP03}J.-P. Demailly and T. Peternell,
A Kawamata-Viehweg vanishing theorem on compact K\"ahler manifolds,
J. Differential Geom. 63 (2) (2003) 231-277.

\bibitem{FavreJonsson}C. Favre and M. Jonsson,
Valuations and multiplier ideals,
J. Amer. Math. Soc. 18 (2005),
no. 3, 655-684.

\bibitem{EM16}L.D. Edholm and J.D. McNeal, The Bergman projection on fat Hartogs triangles: $L^p$ boundedness, Proc. Amer. Math. Soc., 144(5):2185-2196, 2016.

\bibitem{FoW18}J.E. Forn{\ae}ss and J.J.  Wu,
A global approximation result by Bert Alan Taylor and the strong openness conjecture in $\mathbb{C}^n$,
J. Geom. Anal. 28 (2018), no. 1, 1-12.

\bibitem{FoW20}J.E. Forn{\ae}ss and J.J. Wu, Weighted approximation in $\mathbb{C}$, Math. Z. 294 (2020), no. 3-4, 1051-1064.


\bibitem{G-R}H. Grauert and R. Remmert, Coherent analytic sheaves, Grundlehren der mathematischen Wissenchaften, 265, Springer-Verlag, Berlin, 1984.

\bibitem{G2018}Q.A. Guan,
Genral concavity of minimal $L^2$ integrals related to multiplier
sheaves,
arXiv:1811.03261.v4     [math.CV].

\bibitem{G16}Q.A. Guan,
A sharp effectiveness result of Demailly's strong openness conjecture,
Adv. Math. 348 (2019): 51-80.


\bibitem{GM}
Q.A. Guan and Z.T. Mi, Concavity of minimal $L^2$ integrals related to multiplier ideal sheaves, Peking Math. J. 6, 393-457 (2023). https://doi.org/10.1007/s42543-021-00047-5.

\bibitem{GM_Sci} Q.A. Guan and Z.T. Mi, Concavity of minimal $L^2$ integrals related to multiplier ideal
sheaves on weakly pseudoconvex K\"ahler manifolds, Sci. China Math. 65 (2022), no. 5, 887-932.


\bibitem{GMY-concavity2}Q.A. Guan, Z.T. Mi and Z. Yuan, Concavity property of minimal $L^2$ integrals with Lebesgue measurable gain \uppercase\expandafter{\romannumeral2}, https://www.researchgate.net/publication/354464147.


\bibitem{GY-support}Q.A. Guan and Z. Yuan,  An optimal support function related to the strong openness property,  J. Math. Soc. Japan 74 (2022), no. 4, 1269-1293.


\bibitem{GY-concavity}Q.A. Guan and Z. Yuan, Concavity property of minimal $L^2$ integrals with Lebesgue measurable gain, Nagoya Math. J. (2023) 252, 842-905, doi: 10.1017/nmj.2023.12.



\bibitem{guan-zhou12}Q.A. Guan and X.Y. Zhou, Optimal constant problem in the $L^{2}$ extension theorem, C. R. Acad. Sci. Paris. Ser. I. 350 (2012), no. 15--16, 753--756.

\bibitem{guan-zhou13p}Q.A. Guan and X.Y. Zhou, Optimal constant in an $L^2$ extension problem and a proof of a conjecture of Ohsawa,  Sci. China Math., 2015, 58(1): 35--59.

\bibitem{GZSOC}Q.A. Guan and X.Y Zhou,
A proof of Demailly's strong openness conjecture, Ann. of Math.
(2) 182 (2015), no. 2, 605-616.


\bibitem{guan-zhou13ap}Q.A. Guan and X.Y. Zhou, A solution of an $L^{2}$ extension problem with an optimal estimate and applications,
Ann. of Math. (2) 181 (2015), no. 3, 1139--1208.

\bibitem{GZeff}Q.A. Guan and X.Y Zhou,
Effectiveness of Demailly's strong openness conjecture and
related problems, Invent. Math. 202 (2015), no. 2, 635-676.




\bibitem{GZ20}Q.A. Guan and X.Y. Zhou, Restriction formula and subadditivity property related to multiplier ideal sheaves, J. Reine Angew. Math. 769, 1-33 (2020).

\bibitem{Guenancia}H. Guenancia,
Toric plurisubharmonic functions and analytic adjoint ideal sheaves, Math. Z. 271 (3-4) (2012) 1011-1035.



\bibitem{hormander}L. H\"ormander, An  introduction  to  complex  analysis  in  several  variables, 3rd edition (North-Holland, Amsterdam, 1990).


\bibitem{JM12}M. Jonsson and M. Musta\c{t}\u{a}, Valuations and asymptotic invariants for sequences of ideals, Annales de l'Institut Fourier A. 2012, vol. 62, no.6, pp. 2145--2209.
\bibitem{JM13}M. Jonsson and M. Musta\c{t}\u{a}, An algebraic approach to the openness conjecture of Demailly and Koll\'{a}r,
J. Inst. Math. Jussieu (2013), 1--26.



\bibitem{K16}D. Kim, Skoda division of line bundle sections and pseudo-division, Internat. J. Math. 27 (2016), no. 5, 1650042, 12 pp.


\bibitem{KS20}D. Kim and H. Seo, Jumping numbers of analytic multiplier ideals (with an appendix by Sebastien Boucksom), Ann. Polon. Math., 124 (2020), 257-280.


\bibitem{Lazarsfeld1}R. Lazarsfeld,
Positivity in Algebraic Geometry. \uppercase\expandafter{\romannumeral1}. Classical Setting: Line Bundles and Linear Series. Ergebnisse der Mathematik und ihrer Grenzgebiete. 3. Folge. A Series of Modern Surveys in Mathematics [Results in Mathematics and Related Areas. 3rd Series. A Series of Modern Surveys in Mathematics], 48. Springer-Verlag, Berlin, 2004.


\bibitem{Lazarsfeld2}R. Lazarsfeld,
Positivity in Algebraic Geometry. \uppercase\expandafter{\romannumeral2}. Positivity for vector bundles, and multiplier ideals. Ergebnisse der Mathematik und ihrer Grenzgebiete. 3. Folge. A Series of Modern Surveys in Mathematics [Results in Mathematics and Related Areas. 3rd Series. A Series of Modern Surveys in Mathematics], 49. Springer-Verlag, Berlin, 2004.


\bibitem{M-V15}J. D. McNeal,  D. Varolin, $L^2$ estimates for the $\bar\partial$ operator. Bull. Math. Sci. 5 (2015), no. 2, 179--249.

\bibitem{Nadel}A. Nadel,
Multiplier ideal sheaves and K\"ahler-Einstein metrics of positive scalar curvature, Ann. of Math. (2) 132 (3) (1990) 549-596.

\bibitem{Sha15}M.C. Shaw, The Hartogs triangle in complex analysis, In Geometry and topology of submanifolds and currents, volume 646 of Contemp. Math., pages 105-115, Amer. Math. Soc., Providence, RI, 2015.

\bibitem{Siu96}
Y.T. Siu,
The Fujita conjecture and the extension theorem of Ohsawa-Takegoshi,
 Geometric Complex Analysis, World Scientific, Hayama, 1996, pp.223-277.

\bibitem{Siu05}Y.T. Siu,
Multiplier ideal sheaves in complex and algebraic geometry,
Sci. China Ser. A 48 (suppl.) (2005) 1-31.

\bibitem{Siu09}Y.T. Siu,
Dynamic multiplier ideal sheaves and the construction of rational curves in Fano manifolds,
Complex Analysis and Digtial Geometry, in: Acta Univ. Upsaliensis Skr. Uppsala Univ. C Organ. Hist., vol.86, Uppsala Universitet, Uppsala, 2009, pp.323-360.

\bibitem{siu00}Y.T. Siu, Extension of twisted pluricanonical sections with plurisubharmonic weight and invariance of semipositively twisted plurigenera for manifolds not necessarily of general type. Complex geometry (G\"{o}ttingen, 2000), 223--277, Springer, Berlin, (2002).

\bibitem{Tian}G. Tian,
On K\"ahler-Einstein metrics on certain K\"ahler manifolds with $C_1(M)>0$, Invent. Math. 89 (2) (1987) 225-246.


\bibitem{xu}C.Y. Xu, A minimizing valuation is quasi-monomial. Ann. of Math. (2) 191 (2020), no. 3, 1003--1030.


\bibitem{ZZ2018}X.Y Zhou and L.F. Zhu,
An optimal $L^2$ extension theorem on weakly pseudoconvex K\"ahler
manifolds,
J. Differential Geom.110(2018), no.1, 135-186.


\bibitem{ZZ2019}
X.Y Zhou and L.F. Zhu,
Optimal $L^2$ extension of sections from subvarieties in weakly pseudoconvex manifolds. Pacific J. Math. 309 (2020), no. 2, 475-510.


\bibitem{ZhouZhu20siu's}X.Y. Zhou and L.F. Zhu, Siu's lemma, optimal $L^2$ extension and applications to twisted pluricanonical sheaves, Math. Ann. 377 (2020), no. 1-2, 675-722.













	







\end{thebibliography}

\end{document}